\documentclass[a4paper,12pt,reqno,twoside]{amsart}

\usepackage{dirtytalk}
\usepackage{cite}

\usepackage{pgfplots}
\usepackage{mathtools}
\usetikzlibrary{math,calc}

\usepackage{amssymb,amsmath}
\usepackage[foot]{amsaddr}
\usepackage[hidelinks]{hyperref}
\usepackage[hmargin=2.5cm,vmargin=2.5cm]{geometry}
\usepackage{enumitem}

\usepackage{tikz}
\usetikzlibrary{automata,arrows,positioning,calc}
\pgfplotsset{compat=1.18}

\usepackage{enumitem}

\DeclareMathOperator{\diam}{diam}
\DeclareMathOperator{\bP}{\mathbb{P}}

\newcommand{\LL}{{\mathrm{LL}}}

\newcommand\bN{{\mathbb N}}
\newcommand\bR{{\mathbb R}}

\newcommand{\hr}{{\hat{r}}}
\newcommand{\hh}{{\hat{h}}}

\newcommand\cC{{\mathcal C}}
\newcommand\cF{{\mathcal F}}

\newcommand\cP{{\mathcal P}}
\newcommand\cE{{\mathcal E}}

\def\bfP{\mathbf{P}}

\numberwithin{equation}{section}

\newtheorem{thm}{Theorem}[section]
\newtheorem{prop}[thm]{Proposition}
\newtheorem{cor}[thm]{Corollary}
\newtheorem{lem}[thm]{Lemma}
\newtheorem{defn}[thm]{Definition}

\theoremstyle{remark}
\newtheorem{rmk}[thm]{Remark}

\begin{document}

\title[Polynomial rates of memory loss]{Improved polynomial rates of memory loss
for nonstationary intermittent dynamical systems}

\author{A.~Korepanov$^1$}
\address{$^1$School of Sciences, Great Bay University, Dongguan, Guangdong, China}
\email{khumarahn@gmail.com}

\author{J.~Lepp\"anen$^2$}
\address{$^2$Department of Mathematics, Tokai University, Kanagawa, 259-1292, Japan}
\email{leppanen.juho.heikki.g@tokai.ac.jp}
\thanks{\textit{Date}:  \today }
\thanks{A.K.~is supported by EPSRC grant EP/V053493/1. J.L.~is supported by JSPS via the project LEADER. 
The authors thank Wael Bahsoun for useful discussions, Loughborough University and Tokai University for hospitality over the course of this work, and anonymous reviewers for many helpful comments and questions.}

\begin{abstract}
We study nonstationary dynamical systems formed by sequential concatenation of nonuniformly expanding maps with a uniformly expanding first return map. Assuming a polynomially decaying upper bound on the tails of first return times that is nonuniform with respect to location in the sequence, we derive a corresponding sharp polynomial rate of memory loss. As applications, we obtain new estimates on the rate of memory loss for random ergodic compositions of Pomeau–Manneville type intermittent maps and intermittent maps with unbounded derivatives.
\end{abstract}

\maketitle


\section{Introduction}

In this paper, we study statistical properties of nonstationary  intermittent dynamical systems.
By nonstationarity we mean that the rules governing the system's evolution are not fixed, but
they vary with time under the influence of noise, fluctuating environment, control signals, or other
external factors.
This is modeled by time-dependent compositions $T_n \circ \cdots \circ T_1$ along a sequence of transformations
$T_i : X \to X$ acting on a state space $X$, where each $T_i$ describes the rules governing the evolution of states at time $i$.
When an initial state $x_0 \in X$ is sampled from a random distribution, the successive states
$x_n := T_n \circ \cdots \circ T_1(x)$ form a random process that is generally nonstationary. This is in contrast
with the conventional setting of ergodic theory developed in the stationary context of measure preserving
transformations. The statistical theory of nonstationary dynamical systems characterizes the long term behavior of $x_n$, for example through limit theorems for partial sums $\sum_{k=0}^{n-1} v_k(x_k)$ generated by observations $v_k : X \to \bR$ along the time-dependent trajectory.

A classical topic in dynamical systems theory is the decay of correlations, which describes how quickly
initial distributions converge toward an invariant distribution. However, a nonstationary dynamical system may lack any natural invariant distribution. In this context, memory loss, a counterpart to correlation decay, was studied by Ott, Stenlund, and Young~\cite{OSY09}.
Memory is said to be
lost if the time-evolutions $\mu_n, \mu_n'$ of any two sufficiently
regular initial distributions $\mu_0,\mu_0'$ on $X$ satisfy $\Vert \mu_n - \mu_n' \Vert \to 0$ as $n \to \infty$ with
respect to a suitable notion of distance $\Vert \cdot \Vert$. This property can be interpreted as all regular measures
being attracted by the same moving target in a space of measures.
The rate of memory loss has been the subject of several
studies in the dynamical systems literature, including~\cite{MO14, AHNTV15, C21-1, SYZ13, GOT13}, and it plays an important role in further statistical analysis of nonstationary models, as evidenced by studies such as~\cite{S22,HNT17,DFGV18,LS20, DMR22,DH24, DGS24, NSV12}, among others.

Intermittent dynamical systems are chaotic systems that include laminar regions where trajectories can become trapped for long periods of time before transitioning to more chaotic behavior, as seen in the Pomeau--Manneville scenario~\cite{PM80}. The
rate of memory loss for such systems, whether stationary or nonstationary, is known to be polynomial~\cite{AHNTV15,Y99,
LSV99, KL21}.

We now briefly describe the contents of our study.
Building on the coupling approach as in Korepanov, Kosloff, and Melbourne~\cite{KKM19}, in a previous study~\cite{KL21} we analyzed the rate of memory
loss in nonstationary nonuniformly expanding dynamical systems. We considered an abstract framework where the time-dependent trajectory $T_{1,n} := T_n \circ \cdots \circ T_1$ makes frequent returns to a reference set $Y \subset X$, with first return dynamics that exhibit \say{good} distortion properties. Under a polynomially decaying upper bound on the tails of first return times, we derived
memory loss estimates of the form
\begin{align}\label{eq:ml_intro}
| (T_{1,n})_*\mu -  (T_{1,n})_*\nu  | = O ( n^{-\gamma})
\end{align}
for suitable probability measures $\mu$ and $\nu$, where $|\cdot |$ is the total variation distance and $(\cdot)_*$ is the pushforward. The rate coincides with the sharp rate established by Gou\"{e}zel~\cite{G04} in the stationary case, meaning that it cannot be improved without additional information about the sequence $(T_n)$. On the other hand, the estimates in~\cite{KL21} represent
a worst-case scenario where all maps satisfy \emph{uniform} tail bounds of the same polynomial order.
In contrast, the present manuscript derives estimates on the rate of memory loss assuming \emph{non-uniform} tail bounds that can vary with time. This flexibility enables us to address questions such as how the rate of convergence in~\eqref{eq:ml_intro} is influenced by subsequences $(T_{n_k})$ of \say{good} maps, i.e. maps for which the decay rate of stationary correlations is fast in comparison with other maps in the sequence. For an illustration of this phenomenom, see Theorem~\ref{thm:lsv} below.

We apply our abstract result (Theorem~\ref{thm:decdec}) to obtain new bounds on memory loss for random \emph{ergodic} compositions of intermittent interval maps. These results can be viewed as complementary to those of
Bahsoun, Bose, and Ruziboev~\cite{BBR19}, who derived estimates on the quenched decay rate of
correlations for random \emph{i.i.d.}\ compositions of intermittent maps. Their approach, which differs
from ours, was based on the construction of a random Young tower in the spirit of~\cite{BBM02}. Besides the standard example due to Liverani, Saussol, and Vaienti~\cite{LSV99}, we present applications for two families of intermittent maps with unbounded derivatives (see Figure~\ref{fig:maps}),
studied by Cristadoro, Haydn, Marie, and Vaienti in~\cite{CHMV10}, and more recently by
Muhammad and Ruziboev in~\cite{MR24}: Pikovsky maps~\cite{P91} and Grossmann--Horner maps~\cite{GH85}. In our analysis we exploit the distortion bounds proved in~\cite{MR24}.

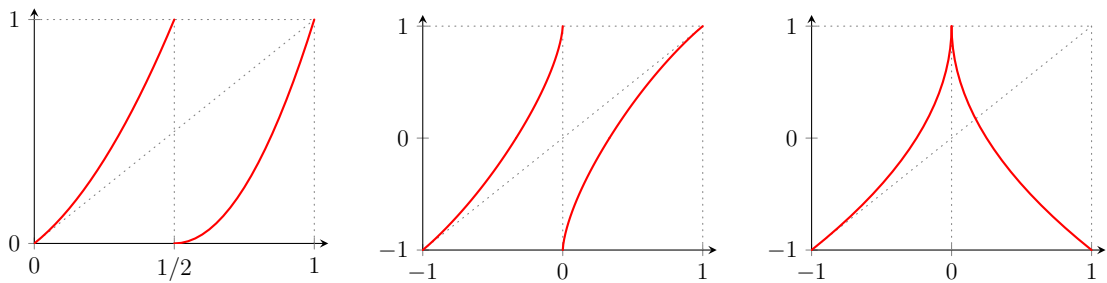
\begin{figure}[h!]\label{fig:maps}
\begin{tikzpicture}[every node/.style={scale=0.75}]
	\tikzmath{
		\ggamma = 1.0;
		\bbeta = 2.0;
	}
	\begin{axis}[
		height       = 1.85in,
		xmax         = 1.05,
		ymax         = 1.05,
		xtick        = {0.0, 0.5, 1.0},
		xticklabels  = {$0$, $1/2$, $1$},
		axis x line  = bottom,
		ytick        = {0.0, 1.0},
		yticklabels  = {$0$, $1$},
		axis y line  = left,
		line cap=round
		]
		\addplot[gray,dotted] coordinates { (0,1) (1,1) };
		\addplot[gray,dotted] coordinates { (1,0) (1,1) };
		\addplot[gray,dotted] coordinates { (0,0) (1,1) };
		\addplot[gray,dotted] coordinates { (0.5,0) (0.5,1) };
		\addplot[red,thick,domain=0.0:0.5, samples=21]{ x * (1 + (2*x)^\ggamma) };
		\addplot[red,thick,domain=0.5:1.0, samples=21]{ 2^\bbeta * (x - 0.5)^\bbeta};
	\end{axis}
\end{tikzpicture}
\quad
\begin{tikzpicture}[every node/.style={scale=0.75}]
	\tikzmath{
		\ggamma = 1.5;
	}
	\begin{axis}[
		height       = 1.85in,
		xmin         = -1.0,
		xmax         = 1.1,
		ymin         = -1.0,
		ymax         = 1.1,
		xtick        = {-1.0, 0.0, 1.0},
		xticklabels  = {$-1$, $0$, $1$},
		axis x line  = bottom,
		ytick        = {-1.0, 0.0, 1.0},
		yticklabels  = {$-1$, $0$, $1$},
		axis y line  = left,
		line cap=round
		]
		\addplot[gray,dotted] coordinates { (-1,1) (1,1) };
		\addplot[gray,dotted] coordinates { (1,1) (1,-1) };
		\addplot[gray,dotted] coordinates { (-1,-1) (1,1) };
		\addplot[gray,dotted] coordinates { (0,-1) (0,1) };
		\addplot[red,thick,domain=-1.0:0.0, samples=21]({0.5 / \ggamma * (1 + x)^\ggamma}, {x});
		\addplot[red,thick,domain=0.0:1.0, samples=21]({x + 0.5 / \ggamma * (1 - x)^\ggamma}, {x});
		\addplot[red,thick,domain=-1.0:0.0, samples=21]({- 0.5 / \ggamma * (1 + x)^\ggamma}, {-x});
		\addplot[red,thick,domain=0.0:1.0, samples=21]({- x - 0.5 / \ggamma * (1 - x)^\ggamma}, {-x});
	\end{axis}
\end{tikzpicture}
\quad
\begin{tikzpicture}[every node/.style={scale=0.75}]
	\tikzmath{
		\eeta = 1.25;
		\ggamma = 1.5;
		\a = 1.0;
		\b = 1.0;
	}
	\begin{axis}[
		height       = 1.85in,
		xmin         = -1.0,
		xmax         = 1.1,
		ymin         = -1.0,
		ymax         = 1.1,
		xtick        = {-1.0, 0.0, 1.0},
		xticklabels  = {$-1$, $0$, $1$},
		axis x line  = bottom,
		ytick        = {-1.0, 0.0, 1.0},
		yticklabels  = {$-1$, $0$, $1$},
		axis y line  = left,
		line cap=round
		]
		\addplot[gray,dotted] coordinates { (-1,1) (1,1) };
		\addplot[gray,dotted] coordinates { (1,1) (1,-1) };
		\addplot[gray,dotted] coordinates { (0,-1) (0,1) };
		\addplot[gray,dotted] coordinates { (-1,-1) (1,1) };
		\addplot[red,thick,domain=-1.0:1.0, samples=21]({0.25 *(x - 1)^2}, {x});
		\addplot[red,thick,domain=-1.0:1.0, samples=21]({- 0.25 *(x - 1)^2}, {x});
	\end{axis}
\end{tikzpicture}
\caption{Graphs of three intermittent maps. Left: Cui's map in~\eqref{eq:cui}; Middle: Pikovsky map~\eqref{eq:pikovsky};
Right: Grossmann--Horner map~\eqref{eq:horner_def}.}
\end{figure}

In the remainder of this introduction, we illustrate our main theorem by applying it to the example from~\cite{LSV99}. That is, for
$\gamma \in (0,1)$, we consider the map $T : X \to X$ on the unit interval
$X = [0,1]$
defined by
\begin{align}\label{eq:lsv}
	T(x)
	= \begin{cases}
		x( 1 + 2^{\gamma} x^{\gamma} ), &x \in [0, \tfrac12), \\
		2( x - \tfrac12 ),
		&x \in [\tfrac12, 1].
	\end{cases}
\end{align}
Let $\beta \in (0,1)$ and $a_\beta > 2^\beta(\beta+ 2)$.
Following~\cite{LSV99},
we define
the cone
\begin{align*}
	\cC_*(\beta) = \{f\in C((0,1])\cap L^1\,:\,
	& \text{$f\ge 0$, $f$ decreasing,}
	\\
	& \text{$x^{\beta+1}f$
		increasing, $f(x)\le a_\beta x^{-\beta} \int_0^1
		f(x) \, dx$}\}.
\end{align*}
This cone contains densities of invariant absolutely continuous probability measures
for maps~\eqref{eq:lsv} with parameters $\gamma \in [0, \beta]$.

Let $(T_n)$ be a sequence of maps $T_1,T_2, \ldots$ in the family~\eqref{eq:lsv} corresponding to parameters
$\gamma_1, \gamma_2, \ldots$, and suppose that $\sup_{n} \gamma_n \le \gamma_* < 1$. By~\cite[Theorem 1.1]{KL21},
whenever $\mu$ and $\nu$ are measures with H\"{o}lder continuous densities,
\begin{align}\label{eq:ml_general}
	| (T_{1,n})_*(  \mu - \nu )  | = O( n^{- 1/\gamma_* } ), \quad \text{as $n \to \infty$.}
\end{align}
In other words, memory loss occurs at a rate corresponding to the map with parameter $\gamma_*$, which
is optimal for a general sequence $(T_n)$ with $\sup_{n} \gamma_n \le \gamma_*$.
However, the following result shows that if the upper and lower asymptotic densities of \say{good} maps
in the sequence $(T_{1,n})$ are sufficiently close, this rate can be improved:

\begin{thm}\label{thm:lsv} Let $\gamma \in (0,1)$. There exists a constant $\kappa_0 \in (0,1)$ such that the following holds.
	Assume that for some $a > 0$, $\kappa \le \kappa_0$, and $N \in \bN$,
	\begin{align}\label{eq:freq_cond-intro}
		\frac{ \#  \{  1 \le k \le n \: : \:   \gamma_k \le \gamma  \}  }{n} \in [  a(1 - \kappa ), a ( 1 + \kappa )   ]
	\end{align}
	holds for all $n \ge N$.
	Let
	$\mu$ and $\mu'$ be probability measures with H\"{o}lder continuous densities,
	and let $\nu$ be a probability measure on $X$
	with density in $\cC_*(\gamma)$. Then
	\begin{itemize}
		\item[(a)] $| (T_{1,n})_*( \mu - \nu ) |
		= O(n^{- 1/\gamma + 1})$,
		\item[(b)] $| (T_{1,n})_*( \mu - \mu' ) |
		= O(n^{- 1/\gamma})$.
	\end{itemize}
\end{thm}

\begin{rmk} A concrete, although complicated, expression for $\kappa_0$ depending on dynamical constants could be extracted from the proofs of Propositions \ref{prop:tb_lsv} 
and \ref{prop:Stail}. 
\end{rmk}

\begin{rmk}
Condition~\eqref{eq:freq_cond-intro} is satisfied by random ergodic compositions of maps in
\eqref{eq:lsv} provided that maps with parameters below $\gamma$ are sampled with positive probability.
See Corollary~\ref{cor:rds-1} for a precise statement.
\end{rmk}

\begin{rmk}
    An interesting question, for which we do not have an answer, is whether the assumption on the
    upper asymptotic density of good maps is necessary in Theorem~\ref{thm:lsv}.
    It is important for our arguments, however: it makes sure that the good maps are
    distributed sufficiently uniformly (see the proof of Proposition~\ref{prop:tb_lsv}),
    for which the lower bound alone is not sufficient.
\end{rmk}

\subsection*{Structure of the paper} In Section~\ref{sec:results}, we
present our main result, an abstract
version of Theorem~\ref{thm:lsv}. In Section~\ref{sec:app}, we formulate and prove
the applications for intermittent maps mentioned in the Introduction.
Section~\ref{sec:proof} contains the proof of the main result, and
Appendix~\ref{sec:prop_mixing} contains the proof of a proposition stated Section~\ref{sec:results} that is
used in one of the applications.

\section{Abstract setup and results}\label{sec:results}

\subsection{Nonstationary nonuniformly expanding dynamical systems, nonuniform tail bounds}
\label{sec:nnue}

Let $(X,d)$ be a bounded metric space.
We endow $X$ with the Borel sigma-algebra, and we only work with measurable sets.
For each $k \ge 1$, let $Y_k \subset X$ and let $m_k$ be a probability measure on $X$ with $m_k(Y_k) = 1$.
We consider a sequence $T_1, T_2, \ldots$ of transformations $T_i : X \to X$ and denote
$T_{k,\ell} = T_\ell \circ \cdots \circ T_k$. (If $k > \ell$,
then $T_{k, \ell}$ is the identity map.)

For a nonnegative measure $\mu$ on $Y_k$ with density $\rho = d\mu / dm_k$, we denote
by $|\mu|_{\LL, k}$ the Lipschitz seminorm of the logarithm of $\rho : Y_k \to \bR_+$:
\[
    |\mu|_{\LL,k}
    = \sup_{y \neq y' \in Y_k} \frac{|\log \rho(y) - \log \rho(y')|}{d(y,y')}
    ,
\]
with a convention that $\log 0 = -\infty$ and $\log 0 - \log 0 = 0$.

We suppose that there exist constants $\lambda > 1$ and $K > 0$
such that the following assumptions hold.

For $x \in X$, $k \ge 1$, define return time functions $\tau_k : X \to \{1,2,\ldots\}$,
\[
    \tau_k(x)
    = \inf \{n \geq 1 : T_{k,k+n-1} (x) \in Y_{k+n} \}.
\]
First, for each $k \ge 1$, we assume existence of a finite or countable partition $\cP_k$ of $X$,
such that $Y_k$ is $\cP_k$-measurable.
Moreover, for each $a \in \cP_k$ we make the following assumptions:
\begin{enumerate}[label=(NU:\arabic*)]
    \item $m_k(a) > 0$ for all $a \subset Y_k$.
    \item\label{ass:constant} $\tau_k$ is constant on $a$ with value $\tau_k(a)$.
    \item\label{ass:distortion} If $a \subset Y_k$, then the map $F_a = T_{k,k+\tau_k(a)-1} |_a : a \to Y_{k + \tau_k(a)}$ is a bijection,
        and for all $y,y' \in a$,
        \[
            d(F_a(y), F_a(y')) \geq \lambda d(y,y')
            .
        \]
        Further, $F_a$ is nonsingular with log-Lipschitz Jacobian with respect to $m_{k + \tau_k(a)}$:
        \[
            \zeta = \frac{d (F_a)_* (m_k|_a)}{dm_{ k + \tau_k(a) }}
            \quad \text{satisfies} \quad
            |\zeta|_{\LL, k + \tau_k(a)} \leq K
            .
        \]
    \item\label{ass:expansion} For all $x,x' \in a$, with $F_a = T_{k, k + \tau_k(a) -  1} |_a$ as above,
        \[
            \max_{0 \leq j \leq \tau_k(a)} d(T_{k,k+j-1}(x), T_{k,k+j-1}(x'))
            \leq K d(F_a(x), F_a(x'))
            .
        \]
\end{enumerate}
In other words, the first return map $y \mapsto T_{k, k + \tau_k(y) - 1}(y)$
with respect to the sequence $T_k, T_{k+1}, \ldots$ is full branch Gibbs-Markov,
and returns from outside of $Y_k$ have bounded backward expansion.

To quantify mixing, for each $k \ge 1$ we define
\begin{align}\label{eq:nonuniform_tails}
	h^k(n) = m_k(  \tau_k \ge n )
        ,
\end{align}
and assume that there exist $\delta_0 > 1$ and $n_0 \ge 0$ such that:
\begin{enumerate}[label=(NU:\arabic*),resume]
	\item\label{ass:mixing}
	$( (T_{k,k+n-1})_* m_k )(Y_{k+n}) \geq \delta_0$ for every $k \ge 1$ and $n \geq n_0$.
\end{enumerate}

Thus, our setting is a variant of that in~\cite{KL21}, with the crucial difference that we relax the assumption
on tail bounds in~\cite{KL21}, allowing them to depend on $k$.

It is sometimes convenient to use the following to verify~\ref{ass:mixing}:

\begin{prop}\label{prop:mixing} Suppose that (NU:1-4) hold. Then, in order to satisfy~\ref{ass:mixing}, it is sufficient that
the following conditions hold:
	\begin{itemize}
		\item[(i)] there exists $N_\# \ge 1$ such that $\sup_{k,j \ge 1} (T_{k, k+j-1})_*m_k( \tau_{k + j} \ge N_\# ) \le 1/2$.
		\item[(ii)] there exist $\delta_\# > 0$ and coprime integers $p_1,\ldots p_M$ such that
		$m_k(  \tau_k = p_m ) \ge \delta_\#$ for all $k \ge 1$ and $1 \le m \le M$.
	\end{itemize}
Moreover, for condition (i) it is sufficient that $\sup_{k \ge 1} \int \tau_k^p \, d m_k < \infty$ holds for some $p > 1$.
\end{prop}

\begin{proof} The proof is similar to that in the stationary case~\cite[Section 4.1]{KKM19}. For completeness, we provide more details in
Appendix~\ref{sec:prop_mixing}.
\end{proof}

\subsection{Memory loss} We consider a sequence of maps $(T_n)$
satisfying (NU:1-5), equipped with a sequence of tail bounds $(h^k)$.

\begin{prop}
	\label{prop:K}
	There exist constants $0 < K_1 < K_2$, depending only on
	$\lambda$ and $K$, such that for each $k \ge 1$ and
	for each nonnegative measure $\mu$ on $Y_k$
	with $|\mu|_{\LL, k} \leq K_2$,
	\[
	\bigl| (  T_{k, k + \tau_k(a) - 1}  )_* (\mu|_a) \bigr|_{\LL, k + \tau_k(a)} \leq K_1
	,
	\]
	whenever $a \in \cP_k$, $a \subset Y_k$.
	The constants $K_1$, $K_2$ can be chosen arbitrarily large.
\end{prop}

\begin{rmk}
	If $\mu$ is a measure supported on $Y_k$ and $a \in \cP_k$, then
	$(T_{k, k + \tau_k(a) - 1})_* (\mu|_a)$ is supported on $Y_{k + \tau_k(a)}$.
\end{rmk}

\begin{proof}[Proof of Proposition~\ref{prop:K}]
	It is standard, see e.g.~\cite[Proposition~3.1]{KKM19}, that
	\[
	\bigl| ( T_{k, k + \tau_k(a) - 1}  )_* (\mu|_a) \bigr|_{\LL, k + \tau_k(a)}
	\leq K + \lambda^{-1} |\mu|_{\LL, k}
	.
	\]
	We can choose any $K_2 > (1-\lambda^{-1})^{-1} K$ and $K_1 = K + \lambda^{-1} K_2$.
\end{proof}

Fix $K_1$, $K_2$ as in Proposition~\ref{prop:K}. The notion of regularity defined below is a slight modification of the one considered in~\cite{KL21}.

\begin{defn} \label{def:reg}
        Let $k \ge 1$.
	A nonnegative measure $\mu$ on $X$ is called \emph{regular}
	with respect to $T_k, T_{k+1}, \ldots$
        if for every $\ell \geq 1$, setting $a_{k, \ell} = \{ x \in X : \tau_k(x) = \ell\}$,
	\begin{align}\label{eq:regular}
            \bigl| (T_{k, k + \ell - 1} )_* (\mu|_{a_{k,\ell}}) \bigr|_{\LL, k + \ell}
	\leq K_1
	.
	\end{align}
	Given  $r : \{0,1,\ldots\} \to [0,\infty)$,
	we say that $\mu$ has \emph{tail bound} $r$ with respect to $T_k, T_{k+1}, \ldots$
	if for all $n \geq 0$,
	\begin{align}\label{eq:tail_bound}
	\mu \bigl( \{ x \in X : \tau_k(x) \geq n \}\bigr)
	\leq r( n )
	.
	\end{align}
	We say that $\mu$ is regular with tail bound $r$ w.r.t.\ $T_{k}, T_{k+1}, \ldots$, if both~\eqref{eq:regular}
        and~\eqref{eq:tail_bound} are satisfied.
\end{defn}

We gather in the following result some basic properties about regular measures that will be
used  later.

\begin{prop}\label{prop:regular} Let $k \ge 1$ and $n \ge 0$ be arbitrary.
\begin{itemize}
	\item[(a)] The measure $m_k$ is regular with tail bound $r = h^k$ w.r.t.\
	$T_{k}, T_{k+1}, \ldots$,
	and every measure $\mu$ on $Y_k$ with $|\mu|_{\LL, k} \leq K_2$ is regular with tail bound
	$r = \mu(Y_k) e^{ \diam ( X)   K_2} h^k$ w.r.t.\ $T_{k}, T_{k+1}, \ldots$ \smallskip
	\item[(b)] If $\{\mu_j\}$ is a finite or countable collection of measures regular w.r.t.\ $T_k, T_{k+1}, \ldots$,
	then $\mu = \sum_{j} \mu_j$ is regular w.r.t.\ $T_k, T_{k+1}, \ldots$  \smallskip
	\item[(c)] 	If $\mu$ is a regular measure w.r.t.\ $T_k, T_{k+1}, \ldots$,  then both
	$(T_{k, k + n -1})_*\mu$ and $((T_{k, k + n -1})_*\mu)|_{X \setminus Y_{k+n}}$ are regular w.r.t.\ $T_{k + n}, T_{k+ n + 1}, \ldots$
        Moreover, if $n \ge 1$,
	 \[
	\Bigl| \bigl( (T_{k, k +  n - 1})_* \mu \bigr) \big|_{Y_{k+n}} \Bigr|_{\LL, k + n}
	\leq K_1
	.
	\]
\end{itemize}
\end{prop}

\begin{proof} (a) The statements about $m_k$ follow by~\ref{ass:distortion} and the definition of $h^k$, while the statements about $\mu$
follow by Proposition~\ref{prop:K} combined with the estimate $d\mu / d m_k \le \mu(Y_k) e^{  \diam(X) | \mu |_{\LL,k} }$.

\noindent (b) The statement follows from the fact that if $\nu = \sum_{j} \nu_j$, where
$\{\nu_j\}$ is a countable or finite collection of nonnegative measures on $Y_k$,
then
$| \nu |_{\LL,k} \le \sup_{j} | \nu_j|_{\LL,k}$.

\noindent (c) It suffices to treat the case $n = 1$. Let $\mu$ be a regular measure w.r.t.\ $T_k, T_{k+1}, \ldots$
We have to verify that $(T_k)_*\mu$ is regular w.r.t.\ $T_{k+1}, T_{k+2}, \ldots$
To this end, let $\ell \geq 1$ and let $a_{k, \ell}$ be as in Definition~\ref{def:reg}.
Let $\mu' = ((T_k)_* \mu) |_{Y_{k+1}}$.
Then
\begin{align*}
    ( T_{ k + 1, k + 1  + \ell - 1  } )_* [  ( (T_k)_*\mu ) |_{a_{k+1, \ell}}  ]
    & = ( T_{ k, k + 1  + \ell - 1  } )_* [   \mu |_{a_{k, \ell + 1}}  ]
    + ( T_{ k + 1, k + 1  + \ell - 1  } )_* [  \mu' |_{a_{k+1, \ell}}  ]
    \\
    & = \mu_1 + \mu_2
    .
\end{align*}
Since $\mu$ is regular, $|\mu_1|_{\LL,k + 1 + \ell} \leq K_1$, and since $|\mu'|_{\LL, k+1} \leq K_1$, by Proposition~\ref{prop:K},
$|\mu_2|_{\LL,k + 1 + \ell} \leq K_1$ which implies the corresponding bound for the left hand side above, as required.
\end{proof}

With this preparation, our main result can be stated as follows:

\begin{thm}
    \label{thm:decdec}
    Suppose that $\mu$ is a regular probability measure  with tail bound $r$ w.r.t.\
    $T_k, T_{k+1}, \ldots$ for some $k \ge 1$.
    Then there exists a decomposition
    \[
        \mu = \sum_{n=1}^\infty \alpha_{n} \mu_{n},
    \]
    where $\mu_{n}$ are probability measures and $\alpha_n$ are nonnegative constants
    with $\sum_{n \geq 1} \alpha_n = 1$ such that
    $(T_{k, k + n - 1})_* \mu_n = m_{k + n}$ for each $n \geq 1$. The sequence $\alpha_n$
    is fully determined by $K_1$, $K_2$, the system constants ($\diam X$, $K$, $\lambda$, $n_0$, $\delta_0$), and
    the functions $r$ and $h^j$, $j \ge k$.
    In particular, $\alpha_n$ does not depend on $\mu$
    in any other way.
    Moreover, suppose there exist constants $0 < \beta' \le \beta$ with $\beta > 1$, $C_\beta, C_\beta' \ge 1$,
    and a sequence of numbers $\Theta_j \in [0,1)$ such that
    \begin{align}\label{eq:tb_assumption}
    h^j(n) \leq C_\beta (1 \vee ( n - \Theta_j j ) )^{-\beta} \quad \text{and} \quad  r(n) \le C_\beta' (1 \vee ( n - \Theta_k k ) )^{-\beta'}
    \end{align}
    hold for all $j \ge k$ and $n \ge 1$. Then, if $\sup_{j \ge 1} \Theta_j \le \Theta$ holds for a sufficiently small
    $\Theta \in (0,1)$ (depending on $\beta,\beta'$, and the system constants), we have that for all $n \ge 1$,
      \[
    \sum_{j \geq n} \alpha_j
    \leq C C_\beta' ( \Theta_k^* k + 1  )^{\beta'} n^{-\beta'}
    .
    \]
    Here,  $\Theta_k^* = \sup_{j \ge k} \Theta_j$ and
    $C$ is a constant depending only on $n_0, \Theta, C_\beta, \beta, \beta'$ and
    $\diam X$, $K$, $\lambda$, $n_0$, $\delta_0$.
\end{thm}

\begin{rmk} Suppose that $\mu$ and $\nu$ are regular probability measures with the same tail bound $r$
w.r.t.\ $T_{k}, T_{k+1}, \ldots$ Under the assumption~\eqref{eq:tb_assumption}, by Theorem~\ref{thm:decdec} we
obtain the following upper bound on memory loss:
\begin{align*}
| (T_{k, k+ n - 1})_* \mu -  (T_{k, k+ n - 1})_* \nu |  \le 2 \sum_{j > n}  \alpha_j
\le 2C C_\beta' ( \Theta_k^* k  + 1 )^{\beta'} n^{-\beta'}.
\end{align*}
\end{rmk}

\begin{rmk} In Section \ref{sec:app}, we apply Theorem~\ref{thm:decdec} to 
intermittent maps composed along random ergodic sequences of parameters.
In this case, $\Theta_n$ is related to the 
rate of convergence in Birkhoff's ergodic theorem. For further details, see Corollary \ref{cor:rds-1} and Remark \ref{rem:about_Theta}.
\end{rmk}


\section{Applications}\label{sec:app}

\subsection{Liverani--Saussol--Vaienti maps} We consider sequences of maps $T_n : X \to X$
in the family~\eqref{eq:lsv} with $X = [0,1]$. First, we prove
Theorem~\ref{thm:lsv} concerning the rate of memory loss for deterministic compositions, after which we look at applications
for random compositions.

Fix a sequence $(T_k)$ of maps in~\eqref{eq:lsv} with parameters $\gamma_k \in (0,1)$ such that
$$
\sup_{k \ge 1} \gamma_k = \gamma^* < 1.
$$
Let $m$ be the Lebesgue measure on $Y = [1/2,1]$ normalized to probability. Denote by $g_k : X \to [0,1/2]$ and $h : X \to Y$ respectively the left and
right inverse branches of $T_k$. Set
\begin{align*}
	x_n(k) = ( g_{k} \circ \cdots \circ g_{k + n - 1} )(1) \quad \text{and} \quad y_n(k) = (h \circ g_{k + 1} \circ \cdots \circ g_{k + n - 1} ) (1).
\end{align*}
Then, for all $k \ge 1$ and $n \ge 1$, $T_{k, k + n - 1}$ maps both $[ x_{n + 1}(k) , x_{n}(k)  ]$ and
$[ y_{n+1}(k) , y_{n}(k)  ]$ bijectively onto $Y$. For each $k \ge 1$, we define a partition of $X$ by
$$
\cP_k = \{    [ x_{n + 1}(k) , x_{n}(k)  ] \: : \:  n \ge 1 \} \cup  \{    [ y_{n + 1}(k) , y_{n}(k)  ] \: : \:  n \ge 1 \}.
$$
It is known (see e.g.~\cite{AHNTV15, BBR19}) that
\begin{align}\label{eq:estim_xk}
	x_{n}(k) = O(  n^{- 1 / \gamma_*}  ) \quad \text{and} \quad y_{n}(k) - 1/2 = O(  n^{ - 1 / \gamma_*}  ).
\end{align}

Now, set $m_k = m$ and $Y_k = Y$ for all $k \ge 1$. It was shown in~\cite{KL21} that (NU:1-5) are satisfied, and that
every probability measure $\mu$ with density in $\cC_*(\gamma)$ is regular w.r.t.\ $T_k,T_{k+1}, \ldots$ for any $k \ge 1$.
Thus, Theorem~\ref{thm:lsv} follows by combining Theorem~\ref{thm:decdec} with the following result.

\begin{prop}\label{prop:tb_lsv} Let $\gamma \in (0,1)$. Assume that there exist $a > 0$,
	$\kappa \in (0,1)$, and $N \in \bN$ such that for all $n \ge N$,
	\begin{align}\label{eq:freq_cond}
		\frac{ |  \{  1 \le k \le n \: : \:   \gamma_k \le \gamma  \}  | }{n} \in [  a(1 - \kappa), a ( 1 + \kappa )   ].
	\end{align}
	Let $\mu$ be a probability measure with density in $\cC_*(\gamma)$. Then $\mu$ has tail bound
	$r(n) = O(n^{ 1 - 1/ \gamma  })$ with respect to $T_1,T_2,\ldots$ Moreover, for each $k \ge 1$,
	$h^k$ in~\eqref{eq:nonuniform_tails} satisfies $h^k(\ell) \le  C ( 1 \vee ( \ell - \Theta k ) )^{-1 / \gamma}$ where
	$C$ is a constant independent of $k$, and $\Theta = 2 \kappa (1 - \kappa)^{-1}$.
\end{prop}

\begin{proof} Let  $\gamma_{ n_{1} }, \gamma_{n_2}, \ldots$ be the subsequence of all parameters $\gamma_n$
	with $\gamma_n \le \gamma$, which by~\eqref{eq:freq_cond} must be an infinite sequence. Let $\widetilde{T}_j =
	T_{n_j}$. It follows by~\eqref{eq:estim_xk} that $\mu$ has tail bound
	$r(\ell) = O( \ell^{ 1 - 1/\gamma } )$ with respect to the sequence $\widetilde{T}_1, \widetilde{T}_2, \ldots$

	Next, for $k, \ell \ge 1$, denote
	\begin{align*}
		G_\ell(k) &= \# \{ k \le j \le k + \ell - 1 \: : \:
                \gamma_j \le \gamma \}  =  \# \bigl( \{  n_j \: : \: j \in \bN \}  \cap [k, k + \ell - 1] \bigr).
	\end{align*}
	Recall that $\tau_k$ denotes the first return time to $Y$ w.r.t.\ $T_k, T_{k+1}, \ldots$ We observe that
	$$
	\tau_k(x) \ge \ell \implies \widetilde{\tau}_k(x) \ge G_\ell(k),
	$$
	where $\widetilde{\tau}_k$ denotes the first return time to $Y$ w.r.t.\ $T_{  n_{m(k)}  }, T_{n_{m(k)+1}}$, and
	$$
	m(k) = \min \{  j \in \bN \: : \: n_j \ge k \}.$$
	Hence, using~\eqref{eq:freq_cond} we see that $\mu$ has tail bound
	$r(\ell) = O(\ell^{1- 1 / \gamma})$ w.r.t.\ $T_1, T_2, \ldots$:
	\begin{align*}
		\mu( \tau_1 \ge \ell ) \le \mu(  \widetilde{\tau}_1 \ge G_\ell(1) ) = O(  G_\ell(1)^{  1 - 1 / \gamma } )
		= O ( (  \ell a ( 1 - \kappa ) )^{ 1 - 1 / \gamma  } ) = O( \ell^{1 - 1 / \gamma } ).
	\end{align*}
	Moreover,
	$$
	m( \tau_k \ge \ell )
	\le m( \widetilde{\tau}_k \ge G_\ell(k) )
	=  O( G_\ell(k)^{- 1 / \gamma } ),
	$$
	where~\eqref{eq:estim_xk} was used in the last inequality.
	Finally, whenever $k > N$, it follows from~\eqref{eq:freq_cond} that
	\begin{align}\label{eq:estim_gl}
		G_\ell(k) &=G_{  k + \ell -1 }(1) - G_{k - 1}(1) \ge
		( k + \ell - 1 ) a (1 - \kappa ) - (k - 1)a (1 + \kappa )  \notag \\
		& \ge a ( 1 -  \kappa  ) \biggl(  \ell -  \frac{2 \kappa }{ 1 - \kappa } k   \biggr).
	\end{align}
\end{proof}

\begin{rmk} Our method can also be used to treat variations of~\eqref{eq:lsv} where the right branch contracts near a
	critical point,
	as in the following example considered by Cui in~\cite{C21-2}, which involves two parameters $\gamma \in (0,1)$ and $\beta \ge 1$ (see Figure~\ref{fig:maps}):
	\begin{align}\label{eq:cui}
		T(x)
		= \begin{cases}
			x( 1 + 2^{\gamma} x^{\gamma} ), &x \in [0, \tfrac12), \\
			2^\beta( x - \tfrac12 )^{\beta},
			&x \in [\tfrac12, 1].
		\end{cases}
	\end{align}
	The map $T$ preserves an absolutely continuous probability measure if and only if $\gamma \beta < 1$.
	For sequences $T_1,T_2,\ldots$ of maps in~\eqref{eq:cui} with parameters $\gamma_k \le \gamma_*$ and
	$\gamma_k \le \beta_*$ where $\gamma_* \beta_* < 1$, an application of Theorem~\ref{thm:decdec} yields the following estimate
	provided that~\eqref{eq:freq_cond} holds:
	for any two probability measures $\mu$ and $\mu'$ with H\"{o}lder continuous densities,
	$$
	| (T_{1,n})_*( \mu - \mu' ) |
	= O(n^{- 1/\gamma \beta_* }).
	$$
\end{rmk}

\subsubsection{Random compositions}\label{sec:rds} Let $\omega = (\omega_n)$ be a random sequence of parameters sampled from
a probability space $(\Omega, \cF, \bfP) =
( \Omega_0^{\bN}, \cE^{\bN},
\bfP )$, where $\Omega_0 \subset (0,1)$ is an interval, $\cE$ is the Borel
sigma-algebra on $\Omega_0$, and $\bN = \{1,2,\ldots\}$. Denote by $T_{n} \circ \cdots \circ T_{1}$
compositions of maps $T_k$ along the random sequence $\omega$, where each $T_k$ corresponds to
a map with parameter $\omega_k$.
As a consequence of Theorem~\ref{thm:lsv}, we obtain the following result:

\begin{cor}\label{cor:rds-1} Let $\gamma \in \Omega_0$. Let $\mu, \mu'$ and $\nu$ be as in Theorem~\ref{thm:lsv}.
	\begin{itemize}
		\item[(i)] 	Suppose that there exists $b > 0$ such that
		for $\bfP$-a.e.\ $\omega \in \Omega$,
		\begin{align}\label{eq:conv}
			\lim_{n \to \infty} \frac{1}{n} \sum_{k=0}^{n}  \mathbf{1}_{[0, \gamma]}(\omega_k)   = b.
		\end{align}
		Then, for
		$\bfP$-a.e.\ $\omega \in \Omega$, bounds (a) and (b) in
		Theorem~\ref{thm:lsv} hold. \smallskip
		\item[(ii)] Suppose that the following three conditions hold: \smallskip
		\begin{itemize}
			\item[(A1)] the shift map $\sigma : \Omega \to \Omega$, $(  \sigma \omega  )_{n } = \omega_{n+1}$, preserves $\bfP$, \smallskip
			\item[(A2)] $(\sigma, \bfP)$ is ergodic, and \smallskip
			\item[(A3)] $\bfP( \omega_1  \le \gamma) > 0$. \smallskip
		\end{itemize}
		Then, for
		$\bfP$-a.e.\ $\omega \in \Omega$, bounds (a) and (b) in
		Theorem~\ref{thm:lsv} hold.
	\end{itemize}
\end{cor}

\begin{proof}
	Clearly~\eqref{eq:conv} implies~\eqref{eq:freq_cond-intro}, so (i) follows by Theorem~\ref{thm:lsv}. For (ii), let
	$\varphi = \mathbf{1}_{[0, \gamma]} \circ \pi_1$ where $\pi_1(\omega) = \omega_1$ projects onto the first coordinate.
	By Birkhoff's ergodic theorem, we have $$
	\lim_{n \to \infty} \frac{1}{n} \sum_{k=0}^{n}  \mathbf{1}_{[0, \gamma]}(\omega_k)  = \lim_{n \to \infty} n^{-1} \sum_{k=0}^n \varphi \circ \sigma^k = \bfP( \omega_1  \le \gamma) > 0$$ 
	for $\bfP$-a.e.\ $\omega \in \Omega$. Hence,~\eqref{eq:conv} holds with $b = \bfP( \omega_1  \le \gamma)$.
\end{proof}

\begin{rmk}
	The constants in the memory loss bounds of Corollary~\ref{cor:rds-1} depend on $\omega$ through $N = N(\omega)$ in
	Theorem~\ref{thm:lsv}.
\end{rmk}

\begin{rmk} In the setting of Corollary~\ref{cor:rds-1}-(ii), if $\bfP( \omega_1 \le \gamma ) > 0$ holds for all $\gamma > 0$,
	then we obtain $| (T_{1,n})_*( \mu - \mu' ) |
	= O(n^{- 1/\gamma})$ for all $\gamma > 0$, i.e.\ memory loss occurs at a superpolynomial rate.
\end{rmk}

\begin{rmk}\label{rem:about_Theta} In the setting of Corollary~\ref{cor:rds-1}(i), suppose that
	for $\bfP$-a.e.\ $\omega \in \Omega$,
	\begin{align*}
		\biggl| \frac{1}{n} \sum_{k=0}^{n-1}  \mathbf{1}_{[0, \gamma]}(\omega_k)   - b \biggr| \le  \Theta_n,
	\end{align*}
	where $\Theta_n = \Theta_n(\omega) \in \bR_+$ and $\lim_{n \to \infty} \Theta_n = 0$.
	Then, similar to~\eqref{eq:estim_gl}, we find that for sufficiently large $n$ and all $j \ge 1$,
	\begin{align*}
		h^j(n) \le C \biggl(  1   \vee  ( n -    C' \Theta_{j}^* j )   \biggr)^{ - 1/ \gamma }
	\end{align*}
	holds for some constants $C$ and $C'$ (depending on $\omega$),
	where $\Theta_{j}^*  = \sup_{\ell \ge j} \Theta_\ell$.
	By Theorem \ref{thm:decdec}, it follows that 
	for any two probability measures $\mu$ and $\mu'$ with H\"{o}lder continuous densities,
	$$
	|  (T_{k, k+ n - 1})_* \mu  - (T_{k, k+ n - 1})_* \mu'  | = O(  ( \Theta_k^* k  + 1 )^{1 / \gamma } n^{- 1/ \gamma }  )
	$$
	holds for $\bfP$-a.e.\ $\omega \in \Omega$, where the constant is independent of $k$. For example, if $\Theta_j = O(j^{ - \psi })$ for some $\psi \in (0,1]$, then
	$$
	|  (T_{n, n+  \lceil n^a \rceil - 1})_* \mu  - (T_{n, n+ \lceil n^a \rceil - 1})_* \mu'  | = O(  n^{( 1 - \psi - a ) / \gamma  }  )
	$$
	for $\bfP$-a.e.\ $\omega \in \Omega$.
\end{rmk}

In situations where the shift map $\sigma$ does not preserve the measure 
$\bfP$, alternative conditions can be formulated to ensure that~\eqref{eq:conv} is satisfied. The following result 
provides an example in the case of asymptotically mean stationary strongly mixing noise.

Recall that, given two sub-$\sigma$-algebras $\mathcal{U}, \mathcal{V} \subset \mathcal{\cF}$, the $\alpha$-mixing coefficient between $\mathcal{U}$ and $\mathcal{V}$ is defined by
$$
\alpha(\mathcal{U}, \mathcal{V}) = \sup_{ U \in \mathcal{U}, \: V \in \mathcal{V} } | \bfP(U \cap V) - \bfP(U) \bfP(V) |.
$$

\begin{prop}\label{prop:non_inv_p} Let $\gamma \in \Omega_0$. Assume the following:
	\begin{itemize}
	\item[(B1)] There exists a probability measure $\bar{\bfP}$ such that for any bounded measurable function $g : \Omega \to \bR$,
	$$
	\lim_{n \to \infty} \int_{\Omega} n^{-1} \sum_{k=0}^{n-1} g \circ \sigma^k \, d \bfP = \int_\Omega g \, d \bar{\bfP}.
	$$
	\item[(B2)] The strong mixing coefficients
	$$
	\alpha(n) =  \sup_{i \ge 1} \alpha ( \cF_1^i, \cF_{i+n}^\infty  )
	$$
	satisfy
	$$
	\alpha(n) = O( \log^{- \beta} (n)  ) \quad \text{for some $\beta > 1$.}
	$$
	Here, $\cF_1^i$ is the sigma-algebra on $\Omega$ generated by the projections $\pi_1, \ldots, \pi_i$;
	$\cF_{i+n}^\infty$ is the sigma-algebra generated by $\pi_{i+n}, \pi_{i + n + 1}, \ldots$, with $\pi_k(\omega) = \omega_k$. \smallskip 
	\item[(B3)] $\bar{\bfP}( \omega_1 \le \gamma ) > 0$.
\end{itemize}
Then, \eqref{eq:conv} holds.
\end{prop}

\begin{proof} Define $\varphi = \mathbf{1}_{[0, \gamma]} \circ \pi_1$ as in the proof of Corollary \ref{cor:rds-1}. By (B3), it is enough to show that, for 
$\bfP$-a.e. $\omega \in \Omega$, 
\begin{align}\label{eq:to_show_birkhoff}
\lim_{n \to \infty} n^{-1} \sum_{k=0}^{n-1} \varphi \circ \sigma^k =  \bar{\bfP}( \omega_1  \le \gamma).
\end{align}
To this end, we decompose $n^{-1} \sum_{k=0}^{n-1} \varphi \circ \sigma^k - \bar{\bfP}( \omega_1  \le \gamma) = I + II$, where 
$$
I = \int_\Omega  n^{-1} \sum_{k=0}^{n-1}  \varphi \circ \sigma^k \, d \bfP - \int_\Omega \varphi  \, d \bar{\bfP}
$$
and
$$
II = n^{-1}\sum_{k=0}^{n-1} \xi_k, \quad \xi_k = \varphi \circ \sigma^k - \int_\Omega \varphi \circ \sigma^k  \, d \bfP.
$$
By (B1), $I \to 0$ as $n \to \infty$. Moreover, $\sup_{k} \sup_{\omega \in \Omega} |\xi_k(\omega)| \le 2$ and, by (B2),
\begin{align*}
	\alpha( \sigma( \xi_i  :  0 \le i \le j ), \sigma(  \xi_i  :  i \ge j + n )  ) \le \alpha(n) 
	= O( \log^{- \beta} (n)  ).
\end{align*}
Here, $\sigma( \xi_i : 0 \le i \le j )$ and $\sigma( \xi_i : i \ge j + n )$ denote the 
sub-$\sigma$-algebras generated by $\{\xi_i\}_{0 \le i \le j}$ and $\{\xi_i\}_{i \ge j+n}$, respectively. 
Thus, $(\xi_n)$ is a uniformly bounded sequence of $\alpha$-mixing random variables with coefficients decaying at the rate $O(\log^{-\beta}(n))$ with $\beta > 1$.
It follows from \cite[Corollary 1]{S93} that $II \to 0$ as $n \to \infty$, for $\bfP$-a.e. $\omega \in \Omega$. Thus, we obtain 
\eqref{eq:to_show_birkhoff}.
\end{proof}

\subsection{Pikovsky maps}

For $\gamma > 1$, let $T :  X \to X$ be a map on $X = [-1,1]$ defined implicitly by the relations
\begin{align}\label{eq:pikovsky}
	x = \begin{cases}
		\frac{1}{2 \gamma } ( 1 + T(x))^{\gamma},  &0 \le x \le \frac{1}{2 \gamma }, \\
		T(x) + \frac{1}{2 \gamma } (1 - T(x))^{\gamma}, &\frac{1}{2 \gamma } \le x \le 1,
	\end{cases}
\end{align}
and by letting $T(x) = - T(-x)$ for $x \in [-1,0]$. The graph of $T$ is illustrated in Figure~\ref{fig:maps}.
The map has neutral fixed points
at $x = 1,-1$, while at $x = 0$ its derivative becomes infinite. There is an interplay between
the neutrality of the fixed points and the degree of the singularity, which is controlled by the parameter $\gamma$:
\begin{align*}
	&T'(x) \approx C_\gamma x^{ 1/ \gamma - 1 }  \quad \text{as $x \searrow 0$}, \\
	&T'(x) \approx 1 +  \frac12 (1 - x)^{ \gamma - 1 }  \quad \text{as $x \nearrow 1$}.
\end{align*}
Note that $T$ preserves the Lebesgue measure,
and that $T$ reduces to the angle doubling map at the limit $\gamma \searrow 1$.
By~\cite[Proposition 5]{CHMV10}, $T$
is polynomially mixing with rate $n^{ - 1 / ( \gamma - 1) }$ for H\"{o}lder  continuous observables. This result was recently extended in~\cite{MR24} to i.i.d.\ random compositions of maps
in the family~\eqref{eq:pikovsky}.

Let $(T_k)$ be a sequence of maps in~\eqref{eq:pikovsky} with parameters $\gamma_1,\gamma_2, \ldots$ We assume that, for
all $k \ge 1$,
$$
\gamma_k \in [\gamma_-, \gamma_+] \quad \text{where $1 < \gamma_{-} < \gamma_{+} < 3$.}
$$
By applying Theorem~\ref{thm:decdec} in combination with certain technical estimates from~\cite{CHMV10, MR24},
we obtain the following bound on memory loss for deterministic compositions:

\begin{thm}\label{thm:pikovsky} Let $\gamma_{-} \le \gamma \le \gamma_+$. 
There exists $\kappa_0 \in (0,1)$ such that the following holds.
Assume that for some $a > 0$, $\kappa \le \kappa_0$,
	and $N \in \bN$,
	\begin{align}\label{eq:freq_pikovsky}
		\frac{ |  \{  1 \le k \le n \: : \:   \gamma_k \le \gamma  \}  | }{n} \in [  a(1 - \kappa), a ( 1 + \kappa )   ]
	\end{align}
	holds for all $n \ge N$.
	Let
	$\mu$ and $\mu'$ be probability measures on $X$ with H\"{o}lder continuous densities. Then
	\begin{align}\label{eq:ml_pikovsky}
		| (T_{1,n})_*( \mu - \mu' ) |  = O( n^{ - \frac{ 1  }{\gamma - 1} }  ).
	\end{align}
\end{thm}

\begin{rmk} A concrete expression for $\kappa_0$ depending on dynamical constants can be extracted from the proofs of Propositions \ref{prop:tail_pikovsky} 
and \ref{prop:Stail}. 
\end{rmk}

Before proving Theorem~\ref{thm:pikovsky}, we state a corollary for random compositions.
Let $(\omega_n)$ be a random sequence of parameters sampled from a probability space $(\Omega, \cF, \bfP) =
( \Omega_0^{\bN}, \cE^{\bN},
\bfP )$, where $\Omega_0 = [ \gamma_-, \gamma_+ ]$
with $1 < \gamma_{-} < \gamma_{+} < 3$, and $\cE$ is the Borel $\sigma$-algebra on $\Omega_0$.
As in Section~\ref{sec:rds}, denote by $T_n \circ \cdots \circ T_1$ compositions of maps $T_k$ from~\eqref{eq:pikovsky} along the random sequence
$\omega$, where each $T_k$ corresponds to a map with parameter $\omega_k$.

\begin{cor}\label{cor:pikovsky_random} Let $\gamma \in \Omega_0$. Assume the following:
	\begin{itemize}
		\item[(A1)] the law $\bfP$ is preserved by the shift map $\sigma$,
		\item[(A2)] $(\sigma, \bfP)$ is ergodic, and
		\item[(A3)] $\bfP(  \omega_1 \le \gamma  ) > 0$.
	\end{itemize}
        Let $\mu$ and $\nu$ be probability measures on $X$ with  H\"{o}lder continuous densities. Then
	\begin{align}\label{eq:memory_loss_random_pikovsky}
		| (T_{1,n})_*( \mu - \nu ) |  = O( n^{  - \frac{1}{\gamma - 1} }  ).
	\end{align}
\end{cor}

\begin{proof}
	The result follows by combining Theorem~\ref{thm:pikovsky} with Birkhoff's
	ergodic theorem, as in the proof of Corollary~\ref{cor:rds-1}.
\end{proof}

\subsubsection{Proof of Theorem~\ref{thm:pikovsky}}\label{sec:proof_pikovsky} Proving Theorem~\ref{thm:pikovsky} amounts to verifying that
nonstationary compositions of Pikovsky maps~\eqref{eq:pikovsky} fit the abstract framework in Section~\ref{sec:results},
and demonstrating that under~\eqref{eq:freq_pikovsky}, the  tails of return times decay at the correct order in view of~\eqref{eq:ml_pikovsky}.
Both of these properties are essentially a consequence of certain estimates
established in~\cite{CHMV10,MR24}. In particular,~\ref{ass:distortion} follows from
the distortion bounds in~\cite[Lemma 2.11]{MR24}. Further details are provided below.

Let $(T_k)$ be a sequence of maps in the family~\eqref{eq:pikovsky} with parameters $\gamma_k \in [\gamma_-, \gamma_{+}]$,
where $1 < \gamma_{-} < \gamma_{+} < 3$. Let $X_{-} = (-0,1)$ and $X_{+} = (0,1)$. Let $\lambda$ be the
Lebesgue measure on $X$ normalized to probability. Denote by $g_{k, -} : (-1,1) \to X_{-}$ and $g_{k, -} : (-1,1) \to X_{+}$ the left and right inverse branches of $T_k$, respectively.

To define  $\cP_k$ in Section~\ref{sec:results}, we recall from~\cite{CHMV10, MR24} the definitions of certain partitions associated with the sequence $(T_k)$. First, for each $n \ge 1$ and $k \ge 1$, let
\begin{align*}
	\Delta_n^{-}(k) = g_{k,-} ( \Delta_{n-1}^{-}( k + 1 )  ) \quad \text{and} \quad
	\Delta_n^{+}(k) = g_{k,+} ( \Delta_{n-1}^{+}( k + 1 )  ) ,
\end{align*}
where
\begin{align*}
    \Delta_0^{-}(k) = ( g_{k, -}(0), 0  ) = g_{k,-}(X_{+}) \quad \forall k \ge 1,
\end{align*}
and
\begin{align*}
    \Delta_0^{+}(k) = ( 0, g_{k, +}(0)  ) = g_{k,+}(X_{-}) \quad \forall k \ge 1.
\end{align*}
Then $\{ \Delta_n^{-}(k)  \}_{n \ge 1}$ is a (mod $\lambda$) partition of $X_{-}$, $\{ \Delta_n^{+}(k)  \}_{n \ge 1}$ is a (mod $\lambda$) partition of $X_{+}$, and $T_k$ maps $\Delta_n^{\pm}(k)$ bijectively onto $\Delta_{n-1}^{\pm}(k+1)$. Moreover,
$\Delta^+_\ell(k) = ( x_{\ell }^+(k), x_{\ell + 1}^+(k)  )$ and $\Delta^-_\ell(k) = ( x_{\ell + 1 }^-(k), x_{\ell}^-(k)  )$,
 where $x_{\ell}^{+}(k) = ( g_{k,+} \circ \cdots \circ g_{k + \ell - 1, +})(0)$ 
 and $x_{\ell}^{-}(k) = ( g_{k,-} \circ \cdots \circ g_{k + \ell - 1, -})(0)$.

Next, for each $k,n \ge 1$ define
\begin{align*}
	\delta_n^{-}(k) = T_{k}^{-1} (   \Delta_{n-1}^{+}(k+1)   ) \cap \Delta_0^{-}(k) \quad \text{and}  \quad
	\delta_n^{+}(k) = T_{k}^{-1} (   \Delta_{n-1}^{-}(k+1)   ) \cap \Delta_0^{+}(k).
\end{align*}
Then $\{ \delta_n^{-}(k)  \}_{n \ge 1}$  and $\{ \delta_n^{+}(k)  \}_{n \ge 1}$ are (mod $\lambda$) partitions of
$\Delta_0^{-}(k)$ and $\Delta_0^{+}(k)$, respectively. Note that the following maps are all bijective:
\begin{align*}
	&T_k : \delta_n^{-}(k) \to \Delta_{n-1}^{+}(k+1), \quad T_k : \delta_n^{+}(k) \to \Delta_{n-1}^{-}(k+1), \\
	&T_{k, k + n -1} : \delta_n^{-}(k) \to \Delta_{0}^{+}(k+n),  \quad T_{k, k + n -1} : \delta_n^{+}(k) \to \Delta_{0}^{-}(k+n).
\end{align*}

Set
\begin{align}\label{eq:Yk_pikovsky}
Y_k = \Delta_0^{-}(k) \cup \Delta_0^{+}(k).
\end{align}

Let $m_k$ be the Lebesgue measure on $Y_k$ normalized to probability.
For each $k \ge 1$,
we define
a partition $\cP_k$ of $X$ as follows:
\begin{align*}
	\cP_k = \{  \delta_{n}^{-}(k) \cup \delta_{n}^{+}(k)   \}_{n \ge 1}
	\cup \{  \Delta_{n}^{-}( k ) \cup  \Delta_{n}^{+}(k)  \}_{n \ge 1}.
\end{align*}
By definition, $T_{k, k + n - 1}$ maps $\delta_{n}(k) :=  \delta_{n}^{-}(k) \cup \delta_{n}^{+}(k)$
bijectively onto $Y_{k + n}$, and $T_{k, k + n - 1}$ maps
$\Delta_n(k) := \Delta_{n}^{-}( k ) \cup  \Delta_{n}^{+}(k)$ bijectively onto $Y_{k+n}$. For $x \in X$, let
$$
\tau_k(x) = \inf \{  n \ge 1 \: : \:   T_{k, k + n -1}(x)  \in Y_{k + n}   \}.
$$
Then $\tau_k$ is constant on the partition elements of $\cP_k$, so that~\ref{ass:constant} is satisfied. By inspecting the
definition of $T$ in~\eqref{eq:pikovsky}, we see that
$| F_a(y) - F_a(y') | \ge \Lambda |y - y'| > 0$ holds for all $y, y' \in a \subset Y_k$ whenever $a \in \cP_k$
where $F_a = T_{k, k + \tau_k(a) -1}$ and $\Lambda = \Lambda( \gamma_{-}, \gamma_{+} ) > 1$.
Then~\ref{ass:distortion} follows by~\cite[Lemma 2.11]{MR24} provided that $1 < \gamma_{-} < \gamma_{+} < 3$:

\begin{prop}\label{lem:nu3_pikovsky} Assume that $1 < \gamma_{-} < \gamma_{+} < 3$. Then there exists $K > 0$ such that
for any $k \ge 1$ and $a \subset Y_k$, $a \in \cP_k$,
	\begin{align*}
		\zeta = \frac{ d(F_a)_*(m_k |_a) }{dm_{k +  \tau_k(a) }} \quad \text{satisfies} \quad | \zeta |_{\LL, k + \tau_k(a)} \le K.
	\end{align*}
\end{prop}

\begin{proof}
    Suppose that $a = \delta_{n}(k)$. Let $x,x' \in Y_{k + n}$. Then
	\begin{align*}
		\zeta(x) = \frac{ \lambda( Y_{k + \tau_k(a)} ) }{  \lambda( Y_k ) }  \frac{  1  }{  T_{k, k + n - 1}'  (z)  },
	\end{align*}
	where $z = (T_{k, k + n - 1} |_{ \delta_{n}(k)  } )^{-1} (x)$.
        Let $z' = (T_{k, k + n - 1} |_{ \delta_{n}(k)  } )^{-1} (x')$. If $z,z' \in \delta_n^{+}(k)$
	or $z,z' \in \delta_n^{-}(k)$, then by~\cite[Lemma 2.11]{MR24},
	$$
	|\log   \zeta(x) - \log  \zeta(x') | \le K|  T_{k, k+n-1} (z) - T_{k, k+n-1} (z') |  = K |x - x '|.
	$$
	Next, suppose that $z$ and $z'$ belong to different smoothness components, say $z \in \delta_n^{-}(k)$
	and $z' \in \delta_n^{+}(k)$. In this case we can exploit the trick described in~\cite[Remark 1]{CHMV10}. Namely,
	by symmetry of $T_k$
	we have $-z' \in \delta_n^{-}(k)$ and $T_j'(y) = T_j'(-x)$. Since each $T_j$ is an odd function,
	\begin{align*}
	|\log   \zeta(x) - \log  \zeta(x') | &= | \log  T_{k, k + n - 1}'  (z) - \log T_{k, k + n - 1}'  (-z')  | \\
	&\le K|  T_{k, k+n-1} (z) - T_{k, k+n-1} (- z') |  = K |x + x '| \le K( x - x' ).
	\end{align*}
\end{proof}

Assumption~\ref{ass:expansion} is trivially true, since all maps in~\eqref{eq:pikovsky} are expanding.
For~\ref{ass:mixing}, since each of the maps $T_n$ preserves $\lambda$, with $N_\# \geq 1$, we have
$$
(T_{k, k + j -1})_*m_k(  \tau_{k+j} \ge N_\#  ) \le \frac{(T_{k, k + j -1})_* \lambda (  \tau_{k+j} \ge N_\#  )}{  \lambda(Y_k) }
= \frac{ \lambda (  \tau_{k+j} \ge N_\#  )}{  \lambda(Y_k) }
\le C N_\#^{ - 1 / ( \gamma_{+} - 1 ) },
$$
where the last inequality follows from~\eqref{eq:tail_tauk_pikovsky} below, and $C$ is a constant determined by $\gamma_{-}, \gamma_{+}$. Moreover, we have
$
m_k( \tau_k = 1 ) \ge m_k( \delta_1^{-}(k) ) \ge \delta_\#
$
for some constant $\delta_\# > 0$ determined by $\gamma_{-}, \gamma_{+}$. Thus,~\ref{ass:mixing} holds
by Proposition~\ref{prop:mixing}.


\begin{prop}\label{prop:tail_pikovsky}  Assume~\eqref{eq:freq_pikovsky}. Then for each $k \ge 1$, $h^k$ in~\eqref{eq:nonuniform_tails} satisfies
$h^k(\ell) = O(  (  1 \vee  ( \ell - \Theta k )  )^{  - 1  / (  \gamma - 1  )  } )$ where the constant is independent of $k$,
and $\Theta = c_0\kappa (1 - \kappa)^{-1}$ for some absolute constant $c_0 > 0$.
Moreover,
for any probability measure $\mu$ with bounded density, $\mu(  \tau_1 \ge n ) = O( n^{ - 1 / (  \gamma - 1  ) } )$.
\end{prop}

\begin{proof} First, consider a sequence $(T_j)$ of maps in~\eqref{eq:freq_pikovsky}
with parameters $\gamma_j \in [\gamma_{-}, \gamma]$  where $\gamma$ is as in \eqref{eq:pikovsky}.
Then
\begin{align*}
	\lambda( \tau_k \ge n   ) = \lambda \biggl(  \bigcup_{\ell \ge n} \Delta_\ell(k)   \biggr) + \lambda \biggl(  \bigcup_{\ell \ge n} \delta_{\ell}(k)   \biggr).
\end{align*}
Denoting by $x_{\ell}^{+}( \gamma )$ the preimages corresponding to the stationary sequence
$\gamma_1 = \gamma_2 = \cdots = \gamma$, we have
$$
1 -  x_{\ell }^{+}(k) \le 1 -  x_{ \ell }^{+}(\gamma) \le  Cn^{ - \frac{1}{\gamma - 1}  },
$$
where the last inequality is due to~\cite{CHMV10}. Hence, $
\lambda (  \cup_{\ell \ge n} \Delta_\ell^{+}(k)   ) = \lambda ( ( x_{n-1}^{+}(k) , 1  ) ) = O(n^{ -  1 / ( \gamma - 1 )}  )
$. The same estimate holds with $\Delta_\ell^{-}(k)$ in place of  $\Delta_\ell^{+}(k)$, leading to
$\lambda (  \cup_{\ell \ge n} \Delta_\ell(k)   ) =  O(n^{ -  1 / ( \gamma - 1 )}  )$. On the other hand, by~\cite[Lemma 2.4]{MR24},
$$
 \lambda \biggl(  \bigcup_{\ell \ge n} \delta_{\ell}^{\pm}(k)   \biggr) = O\biggl(  n^{ - \frac{\gamma_-}{ \gamma  - 1 }  }   \biggr),
$$
where the constant is determined by $\gamma_{-}, \gamma$. Hence,
\begin{align}\label{eq:tail_tauk_pikovsky}
	\lambda( \tau_k \ge n   ) = O(n^{ - \frac{1}{\gamma - 1}  }).
\end{align}

Next, let $(T_j)$ be a sequence of maps as in Theorem~\ref{thm:pikovsky}. Similar to the proof of Theorem~\ref{thm:lsv}, we let  $\widetilde{T}_j =
T_{n_j}$ where $\gamma_{ n_{1} }, \gamma_{n_2}, \ldots$ is the subsequence of all parameters $\gamma_{n}$
with $\gamma_{n} \le  \gamma$ and,
for each $k, \ell \ge 1$, define
\begin{align*}
	G_\ell(k) &= \# \{ k \le j \le k + \ell - 1 \: : \:
        \gamma_j  \le \gamma  \} =  \# \bigl( \{ n_j \: : \: j \in \bN \} \cap [k , k + \ell - 1 ] \bigr).
\end{align*}
Further, let
$$
\widetilde{\tau}_k(\ell) = \inf \{  \ell \ge 1 \: : \:   \widetilde{T}_{ m(k), m(k) + \ell - 1  }(x) \in Y_{  n_{ m(k) + \ell   }  }  \}, \quad
m(k) = \min \{  j \ge 1 \: : \: n_j \ge k  \},
$$
where $\widetilde{T}_{ m(k), m(k) + \ell - 1  } = \widetilde{T}_{ m(k) + \ell - 1  } \circ \cdots \circ \widetilde{T}_{ m(k)  }$.
By the definitions of $\Delta_n(k)$ and $\delta_n(k)$, we have
$$
\bigcup_{  j \ge \ell  } \Delta_j(k) \subset \{   \widetilde{\tau}_k \ge G_\ell(k)  \} \quad \text{and} \quad
\delta_j(k) \subset T_{k}^{-1}  \Delta_{j - 1}(k+1).
$$
Hence, using~\eqref{eq:tail_tauk_pikovsky} we obtain
\begin{align*}
	\mu(  \tau_k \ge \ell  ) &\le \mu(  \widetilde{\tau}_k \ge G_\ell(k) ) + (T_k)_*\mu( \cup_{ j \ge \ell  }  \Delta_{j-1}(k+1)   )  \\
	&\le \mu(  \widetilde{\tau}_k \ge G_\ell(k) ) + (T_k)_*\mu(   \widetilde{\tau}_{k + 1} \ge G_{ \ell - 1}(k + 1)    ) \\ 
	&\le C \lambda(  \widetilde{\tau}_k \ge G_\ell(k) ) + C \lambda(   \widetilde{\tau}_{k + 1} \ge G_{ \ell - 1}(k + 1)    ) \\
	&\le C' G_\ell(k)^{  - \frac{1}{\gamma - 1}   } + C' G_{ \ell - 1}(k + 1)^{   - \frac{1}{\gamma - 1}    },
\end{align*}
where $C' > 0$ is a constant depending on $\gamma, \gamma_{-}, \gamma_{+}$ and $\mu$.
Finally, as in~\eqref{eq:estim_gl}, by invoking~\eqref{eq:freq_pikovsky} we see that
$$
G_\ell(k) \ge \frac{a ( 1 -  \kappa  ) }{2} \biggl(  \ell -  \frac{c_0 \kappa }{ 1 - \kappa } k   \biggr),
$$
whenever $k \ge N$ and $\ell$ is large enough (depending on $\kappa$), where $c_0 > 0$ is an absolute constant.
\end{proof}

To complete the proof of Theorem~\ref{thm:pikovsky}, it remains to verify that measures with H\"{o}lder
continuous densities are regular.

\begin{prop}\label{prop:regular_pikovsky} Let $\mu$ be a probability measure 
on $X$ with a positive H\"{o}lder continuous density. Then, for sufficiently large $K_1$ in the definition of regularity, $\mu$ is regular 
w.r.t.\ $T_k,T_{k+1}, \ldots$ for all $k \ge 1$.
\end{prop}

\begin{proof}
Let $a \in \cP_k$. Denote by $\rho$ the density of $\mu$. First, assume that $\rho$ is Lipschitz continuous
with $\inf \rho = b > 0$. Then, with $F_a$ as in~\ref{ass:distortion},
\begin{align*}
	\rho_a(y) :=  \frac{ d(F_a)_*(  \mu |_a  )  }{  dm_{ k + \tau_k(a) }  }(y) = 
	C_{k,a}
	\frac{ \rho(y_a) }{ F_a'( y_a )  },
\end{align*}
where $y_a = F_a^{-1}(y)$ and $C_{k,a}$ is a scaling constant. Hence, for $y,z \in Y_{k + \tau_k(a)}$, we have
\[
    | \log \rho_a(y) - \log \rho_a(y') | \le I + II,
\]
where
\[
I = | \log \rho(y_a) - \log \rho(z_a)  |
\qquad \text{and} \qquad
II = |  \log F_a'(y_a) -  \log F_a'(z_a)  |.
\]
Since $F_a$ is distance expanding on $a$,
$$
I \le b^{-1} | \rho(y_a) - \rho(z_a)  | \le b^{-1}  L | y_a - z_a | \le b^{-1}  L | y - z |,
$$
where $L$ is the Lipschitz constant of $\rho$.

By Proposition~\ref{lem:nu3_pikovsky}, $II \le K|y-z|$ holds if $a \subset Y_k$. Then suppose that $a \subset X \setminus Y_k$,
say $a = \Delta_\ell(k)$. If $z_a,y_a \in  \Delta_\ell^{+}(k)$, we have $z_a = T_{k-1}( z_a' )$ and $y_a = T_{k-1}( y_a' )$
for some $z_a', y_a' \in \delta_{\ell + 1}^{-}(k-1)$. (In the case $k = 1$, we let $T_0$ be the map with parameter $\gamma_{-}$).
It follows by~\cite[Lemma 2.11]{MR24} that
$$
II = |   \log T_{k, k + \ell - 1}' \circ T_{k-1} ( y_a' ) -  \log T_{k, k + \ell - 1}' \circ T_{k-1} ( z_a' )   | \le K | y - z |.
$$
The same estimate holds if $z_a,y_a \in  \Delta_\ell^{-}(k)$. If $z_a$ and $y_a$ belong to different components
of  $\Delta_\ell(k)$, we can repeat the argument from the proof of Proposition~\ref{lem:nu3_pikovsky} to again
obtain $II \le K | y - z |$. We conclude that $\rho$ is regular w.r.t.\ $T_k, T_{k+1}, \ldots$

Finally, if $\rho$ is H\"{o}lder continuous with exponent
$\eta \in (0,1]$, then $\mu$ is regular with respect to the metric $d_\eta(x,y) = |x - y|^\eta$. Therefore, we can deal with this case
by simply changing the metric from $d(x,y) = |x - y|$ to $d_\eta$ in the definitions of Section~\ref{sec:nnue}.
\end{proof}

By Propositions \ref{prop:tail_pikovsky} and \ref{prop:regular_pikovsky}, the conclusion of Theorem \ref{thm:pikovsky} holds if the densities of 
$\mu$ and $\mu'$ are positive. Finally, the observation that 
\begin{align}\label{eq:decomp_munu}
(T_{1,n})_*( \mu - \mu' ) = 2 (T_{1,n})_*( \mu_1 - \mu'_1 ) 
\end{align}
where 
$$
\mu_1 = \frac{\mu + 1}{  2 } \quad \text{and} \quad \mu'_1 = \frac{\mu' + 1}{  2 } 
$$
are measures with positive H\"{o}lder densities, can be used to remove 
the restriction that the densities of $\mu$ and $\mu'$ are positive.

\subsection{Grossmann--Horner maps} We follow the presentation in~\cite{MR24}.
Let $T : X \to X$ be a map on $X = [-1,1]$ with graph as in Figure~\ref{fig:maps},
having two surjective and symmetric
branches on the subintervals $[-1,0]$ and $[0,1]$, such that $T |_{ X \setminus \{ 0 \} } \in C^1$ and
$T |_{ X \setminus \{ 0, -1, 1 \}   } \in C^2$, defined locally by
\begin{align}\label{eq:horner_def}
	T(x) = \begin{cases}
		1 - b |x|^\eta + o( |x|^\eta ), &\text{if $x \in U_0(\gamma, \eta)$}, \\
		-x + a | x  - 1|^{ \gamma } + o( |x - 1 |^{\gamma } ), &\text{if $x \in U_{1_{-}}(\gamma, \eta)$}, \\
		x + a | x  + 1|^{ \gamma } + o( |x + 1 |^{\gamma } ), &\text{if $x \in U_{-1_{+}}(\gamma, \eta)$}.
	\end{cases}
\end{align}
Here, $a,b > 0$, $\eta > 0$ and $\gamma > 1$ are parameters, $U_0(\gamma, \eta)$ is a small neighborhood of $0$,
and $U_{1_{-}}(\gamma, \eta), U_{-1_{+}}(\gamma, \eta)$ are small neighborhoods of $1,-1$, respectively.
For $i=1,2$, the derivatives of $T$ are assumed to satisfy
\begin{align}\label{eq:horner_diff}
	| T^{(i)}(x) | &= C_{i,0} |x|^{\eta-i} + o( |x|^{\eta -i} ), \quad \text{if $x \in U_0(\gamma, \eta)$}, \notag  \\
	T^{(i)}(x) &= 2  - i + C_{i,1} | x - 1|^{ \gamma - i} + o(|x - 1|^{\gamma - i})  \quad \text{if $x \in U_{1_{-}}(\gamma, \eta)$}, \\
	T^{(i)}(x) &= i - 2 + C_{i,-1} | x + 1|^{ \gamma - i} + o(|x + 1|^{\gamma - i})  \quad \text{if $x \in U_{-1_{+}}(\gamma, \eta)$}, \notag
\end{align}
where $C_{i,j} > 0$ are constants. Further, assume:
\begin{itemize}
	\item $|T'(x)| > 1$ for all $x \neq -1,1$, \smallskip
        \item $T$ is strictly monotonous and convex on the two intervals $[-1,0]$ and $[0,1]$, and \smallskip
	\item the error terms in~\eqref{eq:horner_def} and~\eqref{eq:horner_diff} are uniformly bounded.
\end{itemize}
It was shown in~\cite{CHMV10} that, for $0 < \eta < \min \{  1/( \gamma - 1), 1 \}$, there exists an absolutely continuous invariant
probability measure $\mu$ that mixes polynomially fast for H\"{o}lder continuous observables with rate
$n^{ 1 - 1/ \eta( \gamma - 1 )  }$. A nonautonomous variant of this result was obtained in~\cite{MR24}, who  considered
random i.i.d.\ compositions of maps in~\eqref{eq:horner_def}. The memory loss bound below follows by applying
Theorem~\ref{thm:decdec} in combination with results in~\cite{MR24}.

Fix $0 < a_{-} < a_{+} < \infty$, $0 < b_{-} < b_{+} < \infty$, $1 < \gamma_{-} < \gamma_{+} \le 2$, $0 < \eta_{-} < \eta_{+} < 1$. Let $T_1,T_2,\ldots$ be a sequence of maps in~\eqref{eq:horner_def}, such that
the parameters $a_i,b_i,\eta_i,\gamma_i$ associated to $T_i$ satisfy
\begin{align}\label{eq:horner_params}
	\eta_{-} \le \eta_i \le \eta_{+}, \quad   \gamma_{-} \le \gamma_i \le \gamma_{+}, \quad  a_{-} \le a_i \le a_{+}, \quad b_{-} \le b_i \le b_{+}.
\end{align}
We also assume that the constants $C_{i,j}$ appearing in~\eqref{eq:horner_diff} satisfy the uniform bounds $0 < C_{-} \le C_{i,j} \le C_{+} < \infty$.

\begin{thm}\label{thm:horner} Let $\gamma \le \gamma_+$. 
	There exists $\kappa_0 \in (0,1)$ such that the following holds.
	Assume that for some $a > 0$, $\kappa \le \kappa_0$,
	and $N \in \bN$,
	\begin{align}\label{eq:freq_horner}
		\frac{ |  \{  1 \le k \le n \: : \:   \gamma_k \le \gamma  \}  | }{n} \in [  a(1 - \kappa), a ( 1 + \kappa )   ]
	\end{align}
	holds for all $n \ge N$.
	Let
	$\mu$ and $\mu'$ be probability measures on $X$
	with H\"{o}lder continuous densities. Then,
	\begin{align*}
		| (T_{1,n})_*( \mu - \mu' ) |  = O(  n^{ - 1 / ( \gamma - 1 )  }  ).
	\end{align*}
\end{thm}

\begin{rmk} A concrete expression for $\kappa_0$ depending on dynamical constants can be extracted from the proofs of Propositions  
	\ref{prop:tail_horner}
	and \ref{prop:Stail}. 
\end{rmk}

A bound analogous to~\eqref{eq:memory_loss_random_pikovsky} can be derived by applying Theorem~\ref{thm:horner}
in the case of random ergodic compositions of maps in~\eqref{eq:horner_def}. We omit the exact statement
to avoid repetition.

\subsubsection{Proof of Theorem~\ref{thm:horner}} Following the presentation in~\cite{MR24}, we start by describing
several partitions associated with nonstationary compositions $T_n \circ \cdots \circ T_1$ of maps in~\eqref{eq:horner_def}.
Let $X = [-1,1]$ and,
as in Section~\ref{sec:proof_pikovsky}, denote by $g_{k, -} : (-1,1) \to X_{-}$ and $g_{k, +} : (-1,1) \to X_{+}$ the left and right inverse branches of $T_k$, respectively. Here $X_{-} = (-1, 0)$ and $X_{+} = (0,1)$. Denote by $\lambda$ the Lebesgue measure on $X$
normalized to probability.

For each $n \ge 1$ and $k \ge 1$, let
\begin{align*}
	\Delta_n^{-}(k) = g_{k,-} ( \Delta_{n-1}^{-}( k + 1 )  )
    \quad \text{and} \quad
	\Delta_n^{+}(k) = g_{k,+} ( \Delta_{n-1}^{-}( k + 1 )  ),
\end{align*}
where
\begin{align*}
    \Delta_0^{-}(k) = ( g_{k, -}(0), 0  ) = g_{k,-}(X_{+}) \quad \text{and} \quad
    \Delta_0^{+}(k) = ( 0, g_{k, +}(0)  ) = g_{k,+}(X_{+}) .
\end{align*}
Then $\{ \Delta_n^{-}(k)  \}_{n \ge 1}$ and $\{ \Delta_n^{+}(k)  \}_{n \ge 1}$ are (mod $\lambda$) partitions of $X_{-}$ and
$X_{+}$, respectively. If $n = 0$, $T_k$ maps $\Delta_n^{ \pm }(k)$ bijectively onto $X_+$ and, if $n \ge 1$, $T_k$ maps
$\Delta_n^{ \pm }(k)$ bijectively onto $\Delta_{n-1}^{ - }(k + 1)$.

For $n \ge 1$, let
\begin{align*}
	\delta_n^{-}(k) = T_k^{-1} ( \Delta_{n-1}^{+} (k + 1 )  ) \cap \Delta_0^{-}(k) \quad \text{and} \quad
	\delta_n^{+}(k) = T_k^{-1} ( \Delta_{n-1}^{+} (k + 1 )  ) \cap \Delta_0^{+}(k).
\end{align*}
Then $\{  \delta_n^{-}(k) \}_{n \ge 1}$ and $\{  \delta_n^{+}(k) \}_{n \ge 1}$ are (mod $\lambda$) partitions of
$\Delta_0^{-}(k)$ and $\Delta_0^{+}(k)$, respectively. If $n = 1$, $T_k$ maps $\delta_n^{\pm}(k)$ onto
$\Delta_0^{+}( k + 1)$ and if $n \ge 2$,  $T_{k, k + n - 1}$ maps
$\delta_n^{ \pm }(k)$  bijectively onto $\Delta_0^{-}(k + n)$.

Finally, define (mod $\lambda$) partitions of $\delta_1^{-}(k)$ and $\delta_1^{+}(k)$ respectively by
$$
\delta_{1,n}^{-}(k) = T_{k}^{-1}(  \delta_n^{+}(k + 1)  ) \cap \Delta_0^{-}(k), \quad n \ge 1,
$$
and
$$
\delta_{1,n}^{+}(k) = T_{k}^{-1}(  \delta_n^{+}(k + 1)  ) \cap \Delta_0^{+}(k  ), \quad n \ge 1.
$$
Then $T_{k, k + n}$ maps both $\delta_{1,n}^{-}(k)$ and $\delta_{1,n}^{+}(k)$ bijectively onto
$\Delta_0^{-}( k + n + 1)$.

Let $Y_k = \Delta_0^{-}(k)$, and let $m_k$ be the Lebesgue measure on $Y_k$ normalized to probability.
For each $k \ge 1$, define a partition of $X$ by
$$
\cP_k = \{   \Delta_n^{-}(k), \Delta_n^{+}(k)   \}_{ n \ge 1} \cup \{ \delta_n^{-}(k), \delta_n^{+}(k)  \}_{n \ge 2}
\cup \{  \delta_{1,n}^{-}(k), \delta_{1,n}^{+}(k)   \}_{n \ge 1},
$$
and a return time function by
$$
\tau_k(x) = \inf \{  n \ge 1 \: : \:   T_{k, k + n -1}(x)  \in Y_{k + n}   \}.
$$
Since $\tau_k$ is constant on the elements of $\cP_k$, \ref{ass:constant} holds. Using the convexity
of the two branches together with properties of the first derivatives in~\eqref{eq:horner_diff}, we see that
$| F_a(y) - F_a(y') | \ge \Lambda |y - y'| > 0$ holds for some $\Lambda > 1$ whenever
 $y, y' \in a$ and $a \in \cP_k$.

 \begin{prop}\label{prop:regularity_gh} There exists $K > 0$ such that
 	for any $k \ge 1$ and $a \in \cP_k$,
 	\begin{align*}
 		\zeta = \frac{ d(F_a)_*(m_k |_a) }{dm_{k +  \tau_k(a) }} \quad \text{satisfies} \quad | \zeta |_{\LL, k + \tau_k(a)} \le K.
 	\end{align*}
    Let $\mu$ be a probability measure 
    on $X$ with a positive H\"{o}lder continuous density. Then, for sufficiently large $K_1$ in the definition of regularity, $\mu$ is regular 
    w.r.t.\ $T_k,T_{k+1}, \ldots$ for all $k \ge 1$.
 \end{prop}
 
\begin{rmk} The case of nonnegative H\"{o}lder densities can be treated by using the decomposition \eqref{eq:decomp_munu}.
\end{rmk}

 \begin{proof}[Proof of Proposition \ref{prop:regularity_gh}] The estimate on $ | \zeta |_{\LL, k + \tau_k(a)}$ follows by the distortion bounds in~\cite[Lemma 3.10]{MR24}, using a similar argument as in the proof of Proposition~\ref{lem:nu3_pikovsky}. The regularity of $\mu$ can be shown as in the proof of Proposition~\ref{prop:regular_pikovsky}.
 \end{proof}

 Hence, \ref{ass:distortion} holds. Again~\ref{ass:expansion} is clear since all maps in~\eqref{eq:horner_def} are expanding. For~\ref{ass:mixing}, we observe that, for some constant $\delta_\# > 0$ independent of $k$,
 $$
 m_k( \tau_k = 2 ) \ge | \delta_2^{-}(k) | \ge \delta_\#
 \quad \text{and} \quad m_k( \tau_k = 3 ) \ge | \delta_3^{-}(k) | \ge \delta_\#.
 $$
 Moreover, $\sup_k \int \tau_k^p \, d m_k \leq \sup_k \lambda(Y_k)^{-1} \int \tau_k^p \, d \lambda  < \infty$ holds for some $p > 1$ by Proposition~\ref{prop:tb_unif_gm} below, since
 $\eta_+(  \gamma_+ - 1 ) < 1$.
Now, \ref{ass:mixing} follows by  Proposition~\ref{prop:mixing}.

\begin{prop}
    \label{prop:tb_unif_gm}
    We have the following uniform tail bounds:
    \[
        \sup_{k} m_k( \tau_k \ge n  )
        = O\bigl( n^{   - \frac{ 1 }{ \eta_+( \gamma_+ - 1 )  }   }  \bigr)
        \quad \text{and} \quad 
        \lambda( \tau_k \ge n) = O( n^{ - \frac{1}{ \gamma_+ - 1 }  }  ).
    \]
\end{prop}

\begin{proof} If $\mu$ is a measure on $X$, then 
for $n \ge 2$,
\begin{align}\label{eq:mu_decomp}
	\mu( \tau_k \ge n  ) &= \mu \biggl(  \bigcup_{\ell \ge n} \Delta^{\pm }_\ell(k)   \biggr) +
	 \mu \biggl(  \bigcup_{\ell \ge n} \delta^{\pm }_\ell(k)   \biggr)
	 +  \mu \biggl(  \bigcup_{\ell \ge n - 1 } \delta^{\pm }_{1, \ell } (k)   \biggr) 
	 = I + II + III.
\end{align}
First, suppose that $\mu = m_k$. Then, we have $I = 0$.
By
\cite[Lemma 3.3]{MR24}, $II = O( n^{ - 1 / \eta_+ (  \gamma_+ - 1 ) } )$. Moreover,
\begin{align*}
\lambda \biggl(  \bigcup_{\ell \ge n - 1 } \delta^{ + }_{1, \ell } (k)   \biggr)
\le \lambda \biggl(    T_k^{-1}  \bigcup_{\ell \ge n - 1 } \delta^{ + }_{ \ell } (k + 1)    \biggr)
\le C \lambda \biggl( \bigcup_{\ell \ge n - 1 } \delta^{ + }_{ \ell } (k + 1)  \biggr)
\le C' n^{ - 1 / \eta_+ (  \gamma_+ - 1 ) },
\end{align*}
where $C'$ is independent of $k$. The same estimate holds for $\delta^{ - }_{1, \ell } (k)$ in place of
$\delta^{ + }_{1, \ell } (k)$, so that $III = O( n^{ - 1 / \eta_+ (  \gamma - 1 ) })$.
This gives the desired tail bound for $\mu = m_k$.

Finally, suppose that $\mu = \lambda$. By~\cite[Sublemma 3.2 and equation (53)]{MR24},
$I = O(n^{-1/(\gamma_+ - 1)})$, so the tail bound for $\lambda$ follows.
\end{proof}

To complete the proof of Theorem~\ref{thm:horner}, it remains to show the following.

\begin{prop}\label{prop:tail_horner} Assume~\eqref{eq:freq_horner}. Then, for each $k \ge 1$, $h^k$ in~\eqref{eq:nonuniform_tails} satisfies
\[
    h^k(\ell) = O \bigl(  (  1 \vee  ( \ell - \Theta  k )  )^{  - 1 /  (  \gamma - 1  )  } \bigr),
\]
where the implied constant is independent of $k$, and $\Theta = c_0 \kappa (1 - \kappa )^{-1}$ for some absolute constant $c_0 > 0$.
Further, for any probability measure $\mu$ on $X$ with bounded density,
$\mu( \tau_1 \ge n  ) = O(  n^{-  1 / (  \gamma - 1  ) }  )$.
\end{prop}

\begin{proof}
Let $(T_k)$ be a sequence of maps as in Theorem~\ref{thm:horner}.
As in the proof of Theorem~\ref{thm:lsv}, we let  $\widetilde{T}_j =
T_{n_j}$ where $(T_{n_j})$ is the sub-sequence of maps whose parameters satisfy
$\gamma_{n_j} \le \gamma$. Set
\begin{align*}
	G_\ell(k) &=  \# \bigl( \{ n_j \: : \: j \in \bN \} \cap [k , k + \ell - 1 ]  \bigr), \\
	 \widetilde{\tau}_k(\ell) &= \inf \{  \ell \ge 1 \: : \:   \widetilde{T}_{ m(k), m(k) + \ell - 1  }(x) \in Y_{  n_{ m(k) + \ell   }  }  \}, \\
	 m(k) &= \min \{  j \ge 1 \: : \: n_j \ge k  \}.
\end{align*}
By the definition of $\Delta_n^{-}(k)$ we have $\cup_{  j \ge \ell  } \Delta_j^{-}(k) \subset \{   \widetilde{\tau}_k \ge G_\ell(k)  \}$.
Consequently, using the inclusions
$\Delta_{\ell}^{ + } (k) \subset T_k^{-1} \Delta_{\ell - 1}^{-}(k+1)$, $\delta^{\pm}_\ell(k) \subset T_k^{-1} \Delta_{\ell-1}^{+}(k+1)$,
and $\delta_{1,\ell}^{\pm}(k) \subset T_k^{-1} \delta_\ell^{+} (k + 1)$ together with the assumption that the density of
$\mu$ is bounded, following the proof of Proposition~\ref{prop:tb_unif_gm}, we obtain
\begin{align*}
	\mu(  \tau_k \ge \ell  ) = I + II + III \le C \sum_{j=0}^3  G_{\ell - j }(k + j)^{ - 1/(\gamma - 1)  }  \quad \forall \ell  \ge 3,
\end{align*}
where $C$ is independent of $k$.
Finally, as in~\eqref{eq:estim_gl}, by invoking~\eqref{eq:freq_horner} we see that for some absolute constant $c_0$,
$$
G_\ell(k) \ge a ( 1 -  \kappa  )  \biggl(  \ell -  \frac{c_0 \kappa }{ 1 - \kappa } k   \biggr),
$$
whenever $k \ge N$ and $\ell$ is sufficiently large (depending on $\kappa$).
\end{proof}

\section{Proof of Theorem~\ref{thm:decdec}}\label{sec:proof}

Define $C_h = 2 e^{K_2 \diam X}$ and
\begin{equation}
    \label{eq:hnk}
    h_n^k(\ell)
    = C_h \bigl( h^k(n + \ell) + h^{k + 1}(n + \ell - 1) + \ldots + h^{k + n}(\ell) \bigr)
    .
\end{equation}

\begin{prop}
    \label{prop:onedec} Let $k \ge 1$.
    Suppose that $\mu$ is a probability measure on $Y_k$ with $|\mu|_{\LL,k} \leq K_2$.
    \begin{enumerate}[label=(\alph*)]
        \item\label{prop:onedec:tail}
            For every $n \geq 0$, the measure $(T_{k,k + n - 1})_* \mu$
            is regular with tail bound $\frac{1}{2} h^k_n$
            with respect to $T_{k+n}, T_{k+n+1}, \ldots$ \smallskip
        \item\label{prop:onedec:dec}
            There is a constant $\theta \in (0,1)$, depending only on $K_1$, $K_2$, $\diam X$ and $\delta_0$,
            such that for every $n \geq n_0$,
            \[
                (T_{k,k + n - 1})_* \mu
                = \theta m_{k+n} + (1-\theta) \mu'
                ,
            \]
            where $\mu'$ is a regular probability measure
            with tail bound $h^k_n$ with respect to
            $T_{k+n}, T_{k+n+1}, \ldots$
    \end{enumerate}
\end{prop}

\begin{proof} We prove~\ref{prop:onedec:tail} first.
    Fix $n \geq 0$ and $k \ge 1$. For each $0 \leq j \le n$, define
    \begin{align*}
        Y'_j
        & = \{ y \in Y_{k + n - j } : T_{k + n - j, k + \ell - 1 }(y) \notin Y_{k + \ell }  \text{ for all $n-j <  \ell \le n$}  \}
        ,
        \\
        Y_j''
        & = Y_{k} \cap   T_{k, k + n - j - 1}^{-1} (Y_j')
        \\
        & = \{ y \in Y_k : T_{k, k + n - j - 1}(y) \in Y_{k + n - j}
        \text{ and $T_{k,k + \ell - 1}(y) \notin Y_{k + \ell }$ for all $n - j < \ell \le n$} \}.
    \end{align*}
    In particular, $Y_0' = Y_{k + n}$ and $Y''_0 = \{ y \in Y_k : T_{k, k + n - 1 }(y) \in Y_{k+n} \}$.

    Observe that the sets $Y_j''$ form a partition of $Y_k$, so we can write
    \begin{align}\label{eq:mu_pw_decomp}
        ( T_{k,k + n - 1}  )_*\mu
        = \sum_{j=0}^n ( T_{k,k + n - 1}  )_*\mu_j,
    \end{align}
    where $\mu_j$ is the restriction of $\mu$ to $Y_j''$.

    Next, define
    $\nu_j = (( T_{k, k + n - j - 1} )_*\mu)|_{ Y_{k+n-j} }$ for $0 \le j \le n$,
    and note that for all  $B \subset X$,
    \begin{align*}
        (  T_{k + n - j, k + n - 1} )_* ( \nu_j |_{Y_j'}  ) (B)
        & = \nu_j ( Y_j' \cap T_{k + n - j, k + n - 1}^{-1} (B) ) \\
        &= \mu(   T_{k, k + n - j - 1}^{-1} (Y_j')   \cap T_{k,k + n - 1}^{-1}(B)) \\
        &= (T_{k,k + n - 1} )_*\mu_j(B).
    \end{align*}
	Hence, $(T_{k + n - j, k + n - 1} )_* ( \nu_j |_{Y_j'} ) = (T_{k,k + n - 1} )_*\mu_j$.

    By Proposition~\ref{prop:regular}, $|\nu_j|_{\LL, k + n - j} \leq K_2$. Since $\nu_j(X) \leq 1$, we deduce that
    $\nu_j \leq e^{K_2 \diam X} m_{k + n - j}$ and thus the tail of $\nu_j$ with respect to
    $T_{k + n - j}, T_{k + n - j + 1}, \ldots$ is bounded by $e^{K_2 \diam X} h^{k + n - j}$.
    Observe that $(T_{k + n - j, k + n - 1} )_* ( \nu_j |_{Y_j'} )$
    inherits the tail bound from $\nu_j$ with a time shift, namely
    $(T_{k + n - j, k + n - 1} )_* ( \nu_j |_{Y_j'} )$
    has tail bound $e^{K_2 \diam X} h^{k + n - j}(\cdot + j)$
    with respect to $T_{k + n}, T_{k + n + 1}, \ldots$
    By~\eqref{eq:mu_pw_decomp}, it follows that $(T_{k,k + n - 1})_* \mu$ has tail bound
    $e^{K_2 \diam X} \sum_{j=0}^n h^{k + n - j}(\cdot + j)$
    with respect to $T_{k + n}, T_{k + n + 1}, \ldots$, as required. Moreover,
    $(T_{k,k + n - 1})_* \mu$ is regular w.r.t.\ $T_{k+n}, T_{k+n+1}, \ldots$
    by Proposition~\ref{prop:regular}.

    It remains to prove~\ref{prop:onedec:dec}.
    Let $\theta_0 \in (0,1)$ be such that for every $\theta' \in [0,\theta_0]$,
    every measure $\rho$ on $Y_k$ with $|\rho|_{\LL, k} \leq K_1$  can be written as
    $\rho = \rho(Y_k) \theta' m_k + \rho'_k$ with $|\rho'_k|_{\LL,k} \leq K_2$.
    Such $\theta_0$ exists and only depends on $K_1, K_2$ and $\diam(X)$, see~\cite[Lemma~3.4]{KKM19}.

    Suppose that $n \geq n_0$, and let $\rho_n = \bigl((T_{k,k + n - 1})_* \mu\bigr)\big|_{Y_{k + n}}$.
    By Proposition~\ref{prop:regular}, $\bigl| \rho_n \bigr|_{\LL, k + n} \leq K_2$, and by~\ref{ass:mixing}, $\rho_n(Y_{k+n}) \geq \delta_0 e^{- \diam (X) K_2}$.
    Let $\theta = \min \{  \theta_0 \delta_0 e^{- \diam(X) K_2}  , 1/2\}$.
    Then
    \[
        \rho_n = \theta m_{k + n} + \rho'_{k + n}
        \qquad \text{with } |\rho'_{k+n}|_{\LL, k + n} \leq K_2
        .
    \]
    Define
    \[
        \mu'
        = (1-\theta)^{-1} \bigl( (T_{k,k + n - 1})_* \mu - \theta m_{k + n} \bigr)
        = (1-\theta)^{-1} \Bigl( \rho'_{k+n} + \bigl((T_{k,k+n-1})_*
        \mu \bigr)\big|_{X \setminus Y} \Bigr)
        .
    \]
    Then $\mu'$ is a probability measure and
    $(T_{k,k+n-1})_* \mu = \theta m_{k+n} + (1-\theta) \mu'$.
    By Proposition~\ref{prop:regular},
    $\rho'_{k+n}$ and $\bigl((T_{k,k+n-1})_* \mu \bigr)\big|_{X \setminus Y}$
    are regular measures w.r.t.\ $T_{k+n}, T_{k+n+1}, \ldots$ hence so is $\mu'$.
    To bound the tail of $\mu'$, we note that
    $\mu' \leq (1-\theta)^{-1} (T_{k,k+n-1})_* \mu$ with $(1-\theta)^{-1} \leq 2$
    and apply the bound on the tail of $(T_{k,k + n - 1})_* \mu$ from~\ref{prop:onedec:tail}.
\end{proof}

Fix $\theta$ as in Proposition~\ref{prop:onedec}.  We can translate the decomposition
of a pushforward $(T_{k, k + N -1})_*\mu = \theta m_{ k + N } + (1 - \theta) \mu'$ into a corresponding decomposition of $\mu$
by simply pulling back measures:

\begin{cor}
    \label{cor:onedec}
    Suppose that $N \geq 0$, $k \ge 1$,
    and $\mu$ is a probability measure on $X$
    such that $(T_{k,k+N - 1 })_* \mu$
    is supported on $Y_{k + N}$ and $|(T_{k,k+N - 1})_* \mu|_{\LL, k + N} \leq K_2$.
    Then for every $n \geq N + n_0$,
    \[
        \mu = \theta \mu_n + (1-\theta) \mu'_n
        ,
    \]
    where $\mu_n$, $\mu'_n$ are probability measures with
    $(T_{k,k+n - 1 })_* \mu_n = m_{k+n}$
    and $(T_{k,k+n - 1})_* \mu'_n$ is regular with tail bound $h_{n-N}^{k+N}$
    w.r.t.\ $T_{k + n}, T_{k + n + 1}, \ldots$
\end{cor}

\begin{proof}
    Fix $n \geq N + n_0$.
    Proposition~\ref{prop:onedec} gives the decomposition
    $(T_{k, k + n - 1 })_* \mu =
    ( T_{k + N, k + n - 1} )_* ((T_{k,k + N-1})_* \mu) = \theta m_{k+n} + (1-\theta) \mu'$,
    where $\mu'$ is a regular probability measure with tail bound
    $h_{n-N}^{k+N}$ w.r.t.\ $T_{k + n}, T_{k + n + 1}, \ldots$
    Define $\mu_n$ and $\mu_n'$ by
    \[
        d \mu_n =
        \Bigl( \frac{dm_{k+n}}{d(T_{k,k+n-1})_* \mu} \circ T_{k,k+n-1}
        \Bigr) \, d \mu
        \quad \text{and} \quad
        d \mu'_n = \Bigl( \frac{d\mu'}{d(T_{k,k+n-1})_* \mu}
        \circ T_{k,k+n-1} \Bigr) \, d \mu
        .
    \]
    It is straightforward that $\mu = \theta \mu_n + (1-\theta) \mu'_n$ with
    $(T_{k,k+n-1})_*\mu_n = m_{k+n}$ and $(T_{k,k+n-1})_*\mu'_n = \mu'$.
    This is the desired decomposition.
\end{proof}

\begin{lem}
    \label{lem:mixdec}
    Suppose that $\mu$ is a regular probability
    measure with tail bound $r$ w.r.t.\
    $T_k, T_{k+1}, \ldots$ where $k \ge 1$. Suppose that $r(n)$ is
    nonincreasing, $r(1) = 1$ and $\lim_{n \to \infty} r(n) = 0$.
    Then
    \[
        \mu = \sum_{j=n_0+1}^{\infty} \alpha_j [\theta \mu_j + (1-\theta) \mu'_j ]
        ,
    \]
    where $\alpha_j = r(j-n_0) - r(j+1-n_0)$ and $\mu_j$, $\mu'_j$ are probability measures such that
    $(T_{k,k+j-1})_* \mu_j = m_{k+j}$ and
    $(T_{k,k+j-1})_* \mu'_j$ is regular with tail bound $h_{j}^{k}$
    w.r.t.\ $T_{k+j}, T_{k+j+1}, \ldots$
\end{lem}


\begin{proof}
    Recall the partition $\cP_k$ of
    $X$ corresponding to $T_k, T_{k+1}, \ldots$ For all $n \ge
    1$, define
    \[
        a_{k,n}
        = \cup \{ a \in \cP_k : \tau_k(a) = n \}
        .
    \]
    Let $\nu_n = \mu|_{a_{k,n}}$.
    Then, for each $n \geq 1$, the measure
    $(T_{k,k+n-1})_* \nu_n$ is supported on $Y_{k+n}$ and
    satisfies $|(T_{k,k+n-1})_* \nu_n|_{\LL, k + n} \leq K_1$.
    By Corollary~\ref{cor:onedec}, for each $\ell \geq n + n_0$,
    \begin{equation}
        \label{eq:gggn}
        \nu_n = \mu(a_{k,n}) \bigl[ \theta \nu_{n, \ell} +
        (1-\theta) \nu_{n,\ell}' \bigr]
        ,
    \end{equation}
    where $\nu_{n,\ell}$, $\nu_{n,\ell}'$ are probability measures with
    $(T_{k,k+\ell-1})_* \nu_{n,\ell} = m_{k + \ell}$
    and $(T_{k,k+\ell-1})_* \nu_{n,\ell}'$ is regular
    with tail bound $h_{\ell-n}^{k+n}$ w.r.t.\
    $T_{k + \ell}, T_{k + \ell + 1}, \ldots$

    We observe that 
    \[
        \sum_{n=1}^\ell \mu(a_{k,n}) \geq r(1) - r(\ell + 1)
        \quad \text{ and } \quad
        \sum_{n=1}^\infty \mu(a_{k,n}) = r(1) - \lim_{\ell \to \infty} r(\ell) = 1
        .
    \]
    By~\cite[Proposition~4.7]{KKM19} there exist nonnegative constants
    $\xi_{\ell,n}$, $1 \leq n \leq \ell < \infty$, such that
    \begin{align*}
        \sum_{n = 1}^{\ell} \xi_{\ell,n} \mu(a_{k,n})
        & = r(\ell) - r(\ell+1)
        & & \text{ for each } \ell,
        \\
        \sum_{\ell = n}^\infty \xi_{\ell,n}
        &= 1
        & & \text{ for each } n
        .
    \end{align*}

    For $j \geq n_0 + 1$, define
    $\chi_j = \sum_{n=1}^{j - n_0} \xi_{j-n_0,n} \nu_n$. Then $\mu = \sum_{j=n_0+1}^\infty \chi_j$ and
    $\chi_j(X) = \alpha_j$. From~\eqref{eq:gggn}, we have
    \[
        \chi_j = \alpha_j [ \theta \mu_j + (1-\theta) \mu'_j ]
    \]
    with $\mu_j  = \alpha_j^{-1} \sum_{n=1}^{j-n_0} \xi_{j-n_0, n} \mu(a_{k,n}) \nu_{n,j}$
    and  $\mu'_j = \alpha_j^{-1} \sum_{n=1}^{j-n_0} \xi_{j-n_0, n} \mu(a_{k,n}) \nu'_{n,j}$.
    It is possible that $\alpha_j = 0$, but this does not create problems and we ignore it
    for simplicity.

    It remains to observe that $\mu_j$ and $\mu'_j$
    are probability measures with $(T_{k,k+j-1})_* \mu_j = m_{k+j}$
    and $(T_{k,k+j-1})_* \mu'_j$ is regular with tail
    bound $\max \{ h^{k+ 1}_{j - 1 }, \ldots,
    h^{k+ j - n_0 }_{ n_0 } \} \le h_{j-1}^{k+1} \le h_j^k$
    w.r.t.\ $T_{k+j}, T_{k+j+2}, \ldots$
\end{proof}

Similar to Corollary~\ref{cor:onedec}, we obtain the following result:

\begin{cor}
    \label{cor:mixdec}
    Suppose that $N \geq 0$, $k \ge 1$,
    and $\mu$ is a probability measure such that
    $(T_{k,k + N - 1})_* \mu$ is regular and has tail
    bound $r$ w.r.t.\ $T_{k + N}, T_{k+N+1}, \ldots$, where $r(n)$ is
    nonincreasing, $r(1) = 1$
    and $\lim_{n \to \infty} r(n) = 0$.
    Then
    \begin{align}\label{eq:dec_pullback_ctbl}
        \mu = \sum_{j=n_0+1}^{\infty} \alpha_j [\theta \mu_j + (1-\theta) \mu'_j ]
        ,
    \end{align}
    where $\alpha_j = r(j-n_0) - r(j+1-n_0)$, and $\mu_j$, $\mu'_j$
    are probability measures such that
    $(T_{k,k + N+j - 1})_* \mu_j = m_{k + N + j}$ and
    $(T_{k,k + N + j -1})_* \mu'_j$ is regular with tail bound
    $h^{k+N}_j$ w.r.t.\ $T_{k+N+j}, T_{k+N+j+1}, \ldots$
\end{cor}

Further, we suppose that $r$ is nonnegative with $\lim_{n \to \infty} r(n) = 0$
and define
\[
    \hr (n) = \min \{ 1, r(1), \ldots, r(n)\}
    .
\]
This way, $\hr$ is nonincreasing and $\hr(1) = 1$;
for a probability measure, the tail bounds $r$ and $\hr$ are equivalent.
Similarly define  $\hh^k_n$.

Let $X_1, X_2, \ldots$ be random variables with values in $\{n_0, n_0+1,\ldots\}$, such that
for all $\ell \geq n_0$,
\begin{equation}
    \label{eq:PXj}
    \begin{aligned}
        \bP(X_1 \geq \ell)
    &= \hr(\ell - n_0),
    \\
    \bP(X_{j+1} \geq \ell \mid X_1, \ldots, X_j)
    & = \hh^{k+ X_1 + \ldots + X_{j-1}}_{X_j}(\ell - n_0)
    \quad \text{ for } j \geq 1
    .
    \end{aligned}
\end{equation}
Let $\tau$ be a geometrically distributed random variable with parameter $\theta$,
independent of $\{X_j\}$. That is,
$$
\bP(\tau = \ell) = (1-\theta)^{\ell-1} \theta \quad \forall \ell \in \{1,2,\ldots\}.
$$
Define
\[
    S
    = X_1 + \ldots + X_\tau
    .
\]
Note that $S$ depends on $k$.

\begin{lem}
    \label{lem:proproprobab}
    Suppose that $\mu$ is a regular probability measure with tail bound $r$ w.r.t.\ $T_k, T_{k+1}, \ldots$
    Then there exists a decomposition
    \[
        \mu = \sum_{n=1}^\infty \bP(S = n) \mu_n,
    \]
    where $\mu_n = \mu_n^{(k)}$ are probability measures such that $(T_{k,k+n-1})_* \mu_n = m_{k+n+1}$.
\end{lem}

\begin{proof} By iterating the decomposition in~\eqref{eq:dec_pullback_ctbl},
    we can express $\mu$ as an infinite series
    \begin{align}\label{eq:decompose_mu_series}
        \mu &= \theta \sum_{i > n_0} \alpha_i \mu_i
        + (1-\theta) \theta \sum_{i,j > n_0} \alpha_{i,j} \mu_{i,j}
        + (1-\theta)^2 \theta
        \sum_{i,j, \ell > n_0} \alpha_{i,j,\ell} \mu_{i,j, \ell} + \cdots
    \end{align}
    as follows.

\begin{itemize}
    \item[] \textbf{Step 1}: Apply~\eqref{eq:dec_pullback_ctbl}
    to $\mu$:
    $$
    \mu = \sum_{i > n_0} \alpha_i [
        \theta \mu_i + (1- \theta) \mu_i',
    ]
    $$
    where  $\alpha_i = \hat{r}(i - n_0) - \hat{r}(i + 1 - n_0)$,
        $(T_{k,k+i-1})_* \mu_i = m_{k+i}$, and
        $(T_{k,k+i-1})_* \mu_i'$ is regular with
        tail bound $\hat{h}^k_i$ w.r.t.\
        $T_{k+i}, T_{k+i+1}, \ldots$ \smallskip
    \item[] \textbf{Step 2}: Apply~\eqref{eq:dec_pullback_ctbl} to each
    $\mu_i'$:
    $$
        \mu = \theta \sum_{i > n_0} \alpha_i \mu_i
        + (1-\theta) \theta \sum_{i,j > n_0} \alpha_{i,j} \mu_{i,j}
        + (1-\theta)^2 \sum_{i,j > n_0} \alpha_{i,j} \mu_{i,j}',
    $$
    where $\alpha_{i,j} = \alpha_i ( \hat{h}^k_i(j-n_0) -
    \hat{h}^k_i(j + 1 -n_0) )$, $(T_{k, k+i+j-1})_* \mu_{i,j} = m_{k+i+j}$,
    and $(T_{k, k+i+j-1})_* \mu_{i,j}'$ is regular with
    tail bound $h_j^{k+i}$ w.r.t.\ $T_{k+i+j}, T_{k+i+j + 1}, \ldots$ \smallskip
    \item[] \textbf{Step 3}: Apply~\eqref{eq:dec_pullback_ctbl} to each
    $\mu_{i,j}'$:
    \begin{align*}
        \mu &= \theta \sum_{i > n_0} \alpha_i \mu_i
        + (1-\theta) \theta \sum_{i,j > n_0} \alpha_{i,j} \mu_{i,j}
        + (1-\theta)^2 \theta
        \sum_{i,j,\ell > n_0} \alpha_{i,j,\ell} \mu_{i,j, \ell} \\
        &+ (1-\theta)^3
        \sum_{i,j, \ell > n_0} \alpha_{i,j,\ell} \mu_{i,j,\ell}',
    \end{align*}
    where $\alpha_{i,j, \ell}
    = \alpha_{i,j} ( \hat{h}^{k + j}_i(\ell-n_0) -
    \hat{h}^{k + j}_{i}(\ell + 1 -n_0) )$, $(T_{k, k + i + j +
    \ell - 1})_* \mu_{i,j, \ell} = m_{k + i + j + \ell}$,
    and $(T_{k, k+i+j+\ell-1})_* \mu_{i,j,\ell}'$ is regular with
    tail bound $h_\ell^{k + i + j}$
    w.r.t.\ $T_{k+i+j+\ell}, T_{k+i+j+\ell + 1}, \ldots$  \smallskip
    \item[] And so on...
\end{itemize}

By induction, we find that
for each $n \geq 1$ and $j_1, j_2, \ldots, j_n \geq n_0$,
    \[
        (1-\theta)^{n-1} \theta \alpha_{j_1, j_2, \ldots, j_n}
        = \bP(\tau = n, X_1 = j_1, \ldots, X_n = j_n)
        .
    \]
    Grouping the terms in~\eqref{eq:decompose_mu_series} by the sum of indices,
    we obtain the required decomposition with
    \[
        \mu_n =
        \sum_{\substack{k \geq 1 \\  j_1+\cdots+j_k=n}}
        (1-\theta)^{k-1} \theta
        \alpha_{j_1, \ldots, j_k} \mu'_{j_1, \ldots, j_k}
        \Bigg/
        \sum_{\substack{k \geq 1 \\  j_1+\cdots+j_k=n}}
        (1-\theta)^{k-1} \theta
        \alpha_{j_1, \ldots, j_k}
        .
    \]
\end{proof}

To complete the proof of Theorem~\ref{thm:decdec}, it remains to estimate the tails
$\bP(S \geq n)$, as is done in the following proposition.

\begin{prop}\label{prop:Stail} There exists $\Theta \in (0,1)$ such that the following holds
for any sequence of numbers $\Theta_j \in [0,1)$ with $\sup_{j \ge 1} \Theta_j \le \Theta$.
Suppose that $\mu$ is a
regular probability measure with tail bound $r$ w.r.t.\ $T_k, T_{k+1}, \ldots$ for $k \ge 1$, and that
for some $0 < \beta' \leq \beta$, $\beta > 1$,  $C_\beta, C_\beta' \ge 1$,
$$
h^j(n) \leq C_\beta (1 \vee ( n - \Theta_j j ) )^{-\beta} \quad \text{and} \quad  r(n) \le C_\beta' (1 \vee ( n - \Theta_k k ) )^{-\beta'}
$$
hold for all $j \ge k$ and $n \ge 1$. Then, for all $n \ge 1$,
    \[
        \bP(S \geq n)
        \leq C_\beta'  C ( \Theta_k^* k + 1 )^{ \beta' } n^{- \beta'}
        ,
    \]
    where $\Theta_k^* = \sup_{j \ge k} \Theta_j$, and
    $C > 0$ is a constant depending only on $n_0, \theta, \Theta, C_h, C_\beta, \beta, \beta'$.
\end{prop}

\begin{proof}
    Suppose, without loss of generality, that $h^j$ are nonincreasing, so that $\hh^j_n(\ell) \leq h^j_n(\ell)$.

    Denote $S_j = X_1 + \cdots + X_j$; then $S = S_\tau$.
    We regard $n \geq 1$ and $j \geq 1$ as fixed and write
    \begin{equation}\label{eq:H}
    		\bP(S_{j+1} \geq n) =  \sum_{ \ell = 1}^{n+1} H_\ell,
    \end{equation}
    where
    \begin{align*}
    	H_\ell &= \bP(X_{j+1} \geq n - \ell \mid S_j = \ell) \bP(S_j = \ell), \quad 1 \le \ell \le n, \\
    	H_{n+1} &=  \bP(S_j > n).
    \end{align*}
    From~\eqref{eq:hnk},
    \[
        h^{k + S_{j-1} }_{X_j} \leq h^k_{S_j}
        .
    \]
    Using this inequality together with~\eqref{eq:PXj},
    \[
        \bP(X_{j + 1} \geq n - \ell \mid S_j = \ell)
        \leq h^k_\ell (n - \ell - n_0)
        = C_h \sum_{i = 0}^{\ell } h^{k + i} (n - i - n_0)
        ,
    \]
    where we use the convention $h^i(\ell) = 0$ for $\ell \le 0$.

    To estimate~\eqref{eq:H}, we decompose the sum $\sum_{1 \le \ell \le n + 1} H_\ell$
    intro three parts. To this end, fix $b \in (1/2, 1)$ by
    $b = \bigl(1 + (1 - \theta)^{\frac{1}{2 \beta'}} \bigr) / 2$ and
    let $A \ge 1$ be a large integer whose value will be specified later.
    Restrict to $\Theta$ sufficiently small so that $1 - b - \Theta b > 0$.
    Let $C$ denote various constants which depend only on
    $A, n_0, \theta, \Theta, C_h, C_\beta, \beta, \beta'$. In particular, $C$ is independent of $k$.

    Define
     \[
    R_j = \sup_{ n \ge 1 } n^{\beta'} \bP(S_j \geq n)
    .
    \]
    Suppose that
    \begin{align}\label{eq:restrict_n}
    n > \max \biggl\{ \frac{8 (  A + n_0 )}{ 1 - b - \Theta b  }, b^{-1}  \frac{2 \Theta_k^* k}{1 - b - \Theta b}   \biggr\}.
    \end{align}
    By our assumption on $h^j$,
    \begin{equation}
        \label{eq:A}
        \begin{aligned}
            \sum_{\ell \leq A} H_\ell
            & \leq \sup_{\ell \leq A} \bP(X_{k + 1} \geq n - \ell \mid S_j = \ell)
            \leq C_h \sum_{i \leq A} h^{k + i} (n - i - n_0)
            \\
            & \leq C_h C_\beta A (  n - 2A - n_0 - \Theta_k^* k )^{-\beta} \\
            &\leq C  n^{-\beta}.
        \end{aligned}
    \end{equation}
    In particular,
    \begin{equation}\label{eq:B}
        \sum_{i \leq A} h^{k + i}( n - i - n_0 )
        \leq C n^{-\beta}
        .
    \end{equation}
    Next,
    \begin{equation}
        \label{eq:b}
        \sum_{b n \leq \ell \leq n + 1} H_\ell
        \leq \bP(S_j \geq b n)
        \leq R_j b^{-\beta'} n^{-\beta'}
        .
    \end{equation}
    Further, it is a direct verification that
    \begin{equation}\label{eq:C}
        \sum_{A < \ell \le b n} h^{k+\ell}(n - \ell - n_0)
        \leq C \sum_{A < \ell \le b n} (n - \ell - n_0 -  \Theta_k^*  (k + \ell ) )^{-\beta} \ell^{-\beta'}
        \le A_n n^{-\beta'}
        ,
    \end{equation}
    where $\lim_{A \to \infty , \, n \to \infty} A_n = 0$.
    Now we make the promised choice of $A$, fixing it sufficiently large so that
    \begin{equation}
        \label{eq:AA}
        b^{- \beta'}  + C_h A_n
        \leq (1 - \theta )^{-1/2}
    \end{equation}
    for all $n$ satisfying~\eqref{eq:restrict_n}.

    Using summation by parts together with~\eqref{eq:B} and~\eqref{eq:C},
    \begin{equation}
        \label{eq:Ab}
        \begin{aligned}
            \sum_{A \leq \ell < b n} H_\ell
            & \leq \sum_{A \leq \ell < b n} C_h \sum_{i \leq \ell} h^{k + i} (n - i - n_0) \bP( S_j = \ell )
            \\
            & = - C_h \bP (S_j \geq b n) \sum_{i \le b n} h^{k + i}(n - i - n_0) \\
            &\quad\, + C_h \bP (S_j \geq A) \sum_{i \leq A} h^{k + i}(n - i - n_0) \\
            & \quad\, + C_h \sum_{A < \ell \le  b n} h^{k + \ell } (n - \ell - n_0) \bP(S_j \geq \ell)
            \\
            & \leq C n^{-\beta} + C_h A_n R_j n^{-\beta'}.
        \end{aligned}
    \end{equation}

    Assembling~\eqref{eq:H},~\eqref{eq:A},~\eqref{eq:b} and~\eqref{eq:Ab}, we obtain
    \begin{align*}
        \bP(S_{j+1} \ge n )
        & \leq C n^{- \beta } + R_j n^{- \beta'} (  b^{- \beta'}  + C_h A_n)
        \\
        & \leq C n^{- \beta } + R_j n^{- \beta'} (1 - \theta )^{-1/2}
    \end{align*}
    for all $n$ satisfying~\eqref{eq:restrict_n}, in particular for $n \geq C \Theta_k^* k$.
    Hence for all $n \geq 1$,
    \[
        \bP(S_{j+1} \ge n )
        \le C n^{- \beta } + R_j n^{- \beta'} (1 - \theta )^{-1/2}
        + C ( \Theta_k^* k )^{ \beta' } n^{- \beta'}
        .
    \]
    Then
    \[
        R_{j + 1}
        \leq C + C ( \Theta_k^* k )^{ \beta' }  + (1 - \theta)^{- 1/2} R_j.
    \]
    By induction,
    \begin{align*}
    	R_{j} \le  C_\beta'  C ( \Theta_k^* k + 1 )^{ \beta' }  (1 - \theta)^{- j/2} .
    \end{align*}
    This extends to
    \[
        \bP(S_j \geq n) \leq C_\beta'  C_2 ( \Theta_k^* k + 1 )^{ \beta' }  (1 - \theta)^{- j/2} n^{-\beta'}
        .
    \]

    Since $\tau$ and $S_j$ are independent and $\bP(\tau = j) = (1 - \theta)^{j - 1} \theta $,
    \begin{align*}
        \bP(S \geq n)
        &= \sum_{j \geq 1} \bP(S_j \geq n) \bP(\tau = j) \\ 
        &\leq C_\beta'  C_2 ( \Theta_k^* k + 1 )^{ \beta' } \theta (1 - \theta)^{-1} \sum_{j \geq 1} (1 - \theta)^{j / 2} n^{-\beta'}
        \leq C_\beta'  C_3 ( \Theta_k^* k + 1 )^{ \beta' } n^{- \beta'}
        ,
    \end{align*}
    as desired.
\end{proof}


\appendix

\section{Proof of Proposition~\ref{prop:mixing}}\label{sec:prop_mixing}

Assume that conditions (i) and (ii) in Proposition~\ref{prop:mixing} hold.
Let $N_1 = \max \{  p_1^2, \ldots, p_M^2 \}$. By~\cite{S77}, for all $n \ge N_1$ we can write
$
n = \sum_{j=1}^M a_j p_j
$ for some integers $a_j \ge 0$.  For $N_1 \le N < n$, we have
$$
(T_{k,  k + n - 1})_*m_k  = ( T_{ k + n - N  ,  k + n - 1 } )_*\nu,
$$
where $\nu = (T_{k,  k +  n - N - 1  })_*m_k$. By Proposition~\ref{prop:regular}, $\nu$ is regular w.r.t $T_{k + n - N}, T_{k + n - N + 1}, \ldots$, and
$\bigl| \nu |_{ Y_{k + n - N } } \bigr|_{ \LL, k + n - N } \le K_1$.

Let $N = N_1 + N_\#$. Condition (i) in Proposition~\ref{prop:mixing} reads
\begin{align}\label{eq:uniform_tail_bound}
\sup_{k,j \ge 1} (T_{k, k+j-1})_*m_k( \tau_{k + j} \geq N_\# ) \le 1/2.
\end{align}
Using the regularity of
$\nu$, we obtain
\begin{align*}
	&(T_{k, k + n - 1  })_*m_k(Y_{ k + n   }) \\ 
	&\ge \sum_{ \ell \le N_\#}
	( T_{ k + n - N  , k + n - 1  } )_*( \nu  |_{  \{  \tau_{k + n - N} = \ell   \}  }  ) (Y_{k + n}) \\
	&= \sum_{\ell \le N_\#}
	( T_{ k + n - N  + \ell   , k + n - 1  } )_*  ( T_{ k + n - N ,   k + n - N  +  \ell - 1     } )_* (   \nu  |_{  \{  \tau_{k + n - N} = \ell   \}  }   ) (Y_{k+n})
	\\
	&\ge  \sum_{\ell \le N_\#}  e^{ -  \text{diam}(X)  K_1} (  T_{ k + n - N , k + n - N + \ell - 1  } )_*(  \nu  |_{  \{  \tau_{k + n - N} = \ell   \}  }   ) (  Y_{k + n - N + \ell } ) \\ 
	&\times (T_{k + n - N + \ell, k + n - 1})_* m_{k + n - N + \ell }(Y_{k + n}).
\end{align*}
As in~\cite[Proposition 4.4]{KKM19}, using our assumption $m_k(  \tau_k = p_m ) \ge \delta_\#$
together with Proposition~\ref{prop:regular}, we find that for any $\ell \le N_\#$,
$$
(T_{k + n - N + \ell, k + n - 1})_* m_{k + n - N + \ell }(Y_{k + n}) \ge ( \delta_\# e^{ - \diam(X) K_1  } )^N.
$$
Summing over $\ell$, we arrive at the lower bound 
\begin{align*}
	(T_{k, k + n - 1  })_*m_k(Y_{k + n}) \ge (\delta_\# e^{- \diam(X) K_1  } )^{N+1} \nu(\tau_{k + n - N } \le N_\# ).
\end{align*}
By~\eqref{eq:uniform_tail_bound}, $\nu(\tau_{k + n - N } \le N_\# ) \ge 1/2$. Therefore,
$$
(T_{k, k + n - 1  })_*m_k(Y_{k + n}) \ge  (\delta_\# e^{- \diam(X) K_1  } )^{N+1}/2,
$$
whenever $n > N_1 + N_\#$. Thus, condition~\ref{ass:mixing} holds.

Finally, to see that $\sup_k \int \tau_k^p \, d m_k < \infty$ implies~\eqref{eq:uniform_tail_bound}, suppose that
$j,k \ge 1$. Since $m_k$ is a measure on $Y_k$ with $| m_k |_{\LL, k} \le K_2$, arguing as in the
proof of Proposition~\ref{prop:onedec}\ref{prop:onedec:tail}, write
\begin{align*}
	 (T_{k, k + j -1})_*m_k (   \tau_{k + j }  > N_\#  )  &\le \frac{1}{2} h_{j}^k (N_\# + 1)  \\
	&= \frac{C_h}{2} \biggl(
	m_k( \tau_k \ge j + N_\# + 1  ) + \ldots + m_{k }( \tau_{k + j} \ge N_\# + 1  )
	\biggr).
\end{align*}
Using $\sup_k \int \tau_k^p \, d m_k < \infty$ together with Markov's inequality, we obtain
$$
(T_{k, k + j -1})_*m_k (   \tau_{k + j }  \geq N_\#  ) \le C \sum_{j \geq N_\#} j^{-p},
$$
where $C$ is a constant independent of $k,j$. Hence~\eqref{eq:uniform_tail_bound} follows by choosing
$N_\#$ sufficiently large.


\bigskip
\bigskip
\bibliography{seq-main}{}
\bibliographystyle{plainurl}

\end{document}